\documentclass[10pt]{article}
\usepackage{amsthm}
\usepackage{latexsym}
\usepackage{amsmath}
\usepackage{amssymb}
\usepackage{amsfonts}
\usepackage{graphicx}
\def\curl{\operatorname{curl}}
\def\Div{\operatorname{Div}}
\def\div{\operatorname{div}}

\def\tan{\operatorname{tan}}

\def\claim#1{\begin{trivlist}\item[\hskip\labelsep\bf#1]\it}
\def\endclaim{\end{trivlist}}

\numberwithin{equation}{section}

\newtheorem{theorem}{Theorem}[section]
\newtheorem{lemma}[theorem]{Lemma}
\newtheorem{proposition}[theorem]{Proposition}
\newtheorem{property}{Property}

\theoremstyle{definition}

\newtheorem{remark}[theorem]{Remark}

\begin{document}
\allowdisplaybreaks
  \title{Reconstruction of interfaces using CGO solutions for the Maxwell equations }
%
\author{
Manas Kar\thanks{RICAM, Austrian Academy of Sciences,
Altenbergerstrasse 69, A-4040, Linz, Austria.
(Email:manas.kar@oeaw.ac.at)
\newline
Supported by the Austrian Science Fund (FWF): P22341-N18.} \qquad{ Mourad Sini}
\thanks{RICAM, Austrian Academy of Sciences,
Altenbergerstrasse 69, A-4040, Linz, Austria.
(Email:mourad.sini@oeaw.ac.at)
\newline
Partially supported by the Austrian Science Fund (FWF): P22341-N18.}
}
%
%
%
\maketitle
 
%

\begin{abstract} 
We deal with the problem of reconstructing interfaces using complex geometrical optics solutions for the Maxwell system.
The contributions are twofold.
 
\noindent First, we justify the enclosure method for the impenetrable obstacle case avoiding any assumption on the directions of the phases
of the CGO's (or the curvature of obstacle's surface). In addition, we need only a Lipschitz regularity of this surface. 
The analysis is based on some fine properties of the corresponding layer potentials in appropriate Sobolev spaces.

\noindent Second, we justify this method also for the penetrable case, where the interface is modeled by the jump (or the discontinuity)
of the magnetic permeability $\mu$. A key point of the analysis is the global $L^p$-estimates for the curl of the solutions of the Maxwell system with discontinuous
 coefficients. These estimates are justified here for $p$ near $2$ generalizing to the Maxwell's case the well known Meyers's $L^p$ estimates
of the gradient of the solution of scalar divergence form elliptic problems.
\end{abstract}
\begin{section}{\textbf{Introduction and statement of the results:}}
Let $\Omega$ $\subset$ ${\mathbb R}^3$ be a bounded domain with $C^1$-smooth boundary. Let D be a subset of $\Omega$ with  
Lipschitz boundary and the connected complement ${\mathbb R}^3\setminus
\overline{D}$.
 We are concerned with the electromagnetic wave propagation in an isotropic medium in ${\mathbb R}^3$ with
the electric permittivity $\epsilon > 0$ and the magnetic permeability $\mu > 0$.
 We assume 
 $\epsilon \in W^{1, \infty}(\Omega)$ such that $\epsilon=1$
in $\Omega \setminus{\overline{D}}$. We also assume $\mu(x) := 1 - \mu_D(x)\chi
_D(x)$ to be a measurable function, where $\mu_D \in L^{\infty}(D)$ and $\chi_D$ is the
characteristic function of $D$ such that $\vert\mu_D\vert \geq C>0$.
 If we denote by $E$, $H$ the electric and the magnetic fields respectively, then the first problem we are interested with is the impenetrable obstacle problem 
\begin{equation}\label{pene}
 \begin{cases}
 & \curl E-ikH=0  \ \ \text{in}\ \Omega\backslash\overline{D}, \\
 & \curl H+ikE=0  \ \ \text{in}\ \Omega\backslash\overline{D},\\
 & \nu\wedge E=f \ \ \text{on}\ \partial\Omega,\\
 & \nu\wedge H=0 \ \ \text{on}\ \partial{D},\\
\end{cases}
\end{equation}
and the second one is the penetrable obstacle problem
\begin{equation}\label{im_pene}
\begin{cases} 
& \curl E-ik \mu H=0 \ \ \text{in}\ \Omega,\\
& \curl H+ik \epsilon E=0 \ \ \text{in}\ \Omega,\\
& \nu\wedge E=f \ \ \text{on}\ \partial\Omega,
\end{cases}
\end{equation}
where $\nu$ is the unit outer normal vector on $\partial\Omega\cup\partial{D}$ and $k>0$ is the wave number.
Assume that $k$ is not an eigenvalue for the spectral problem corresponding to \eqref{pene} or \eqref{im_pene}. Then
both the problems \eqref{pene} and \eqref{im_pene} are  well posed in the spaces $H(\curl;\Omega\setminus \overline{D})$ and $H(\curl;\Omega)$ respectively,
see \cite{Monk} and \cite{Mitrea} for instance.
 \end{section}\\
\textbf{Impedance Map:}
We define the impedance map 
$\Lambda_D : {TH^{-\frac{1}{2}}(\partial\Omega)}\rightarrow{TH^{-\frac{1}{2}}(\partial\Omega)}$ for either the exterior or
the interior problems as follows:
\begin{center}
$\Lambda_D({\nu}\wedge {E|_{\partial\Omega}})=({\nu}\wedge {H|_{\partial\Omega}})$,
\end{center}
where $TH^{-\frac{1}{2}}(\partial \Omega) :=\{ f\in H^{-\frac{1}{2}}(\partial \Omega)/ \nu\cdot f= 0 \}.$
This impedance map is bounded. We denote by $\Lambda_\emptyset $ the impedance map for the domain without an obstacle.\\
\textbf{Construction of CGO solutions:}
In \cite{Zhou}, the complex geometrical optic solutions for the Maxwell's equation were constructed as follows.
 Let $\rho, {\rho}^{\bot} \in \mathbb{S}^2$ with $\rho\cdot{\rho}^{\bot}$= 0. 
Given $\theta, \eta \in \mathbb{C}^3$ of the form
\begin{equation}\label{form-eta}
\eta := \frac{1}{\vert \zeta\vert}(-(\zeta\cdot a)\zeta - k\zeta\wedge b + k^2a)\; \mbox{and}\; \theta := 
\frac{1}{\vert \zeta\vert} (k\zeta\wedge a - (\zeta\cdot b)\zeta + k^2b) 
\end{equation}
where $\zeta = -i \tau \rho + \sqrt{\tau^2 + k^2}\rho^{\perp}$ and $a\in\mathbb{R}^3$, $b\in\mathbb{C}^3$,
then for $\tau>0 $ large enough, there exists a unique (complex geometrical 
optic) solution $(E_0,H_0) \in H^1(\Omega)\times H^1(\Omega) $ of Maxwell's equations 
\begin{equation}\label{CGO-linear-Maxwell}
\begin{cases}
  & \curl {E_0}-ik{H_0}=0 \ \ \text{in}\ \Omega, \\
  & \curl {H_0}+ik{E_0}=0 \ \ \text{in}\ \Omega,
\end{cases}
\end{equation}
of the form
\begin{equation}\label{CGOpen}
\begin{cases} 
& E_0 = \eta e^{\{{\tau(x{\cdot}\rho)}+i\sqrt{{\tau}^2+ k^2}{x\cdot{\rho}^\perp}\}}, \\
& H_0 = \theta e^{\{{\tau(x{\cdot}\rho)}+i\sqrt{{\tau}^2+k^2}{x\cdot{\rho}^\perp}\}}.
\end{cases}
\end{equation}
In our work we use special cases of these CGO solutions. Precisely, in the impenetrable obstacle case, we choose $a$ and $b$ such that 
$a=\sqrt{2}\rho^{\perp},b\perp\rho$ and $b\perp\rho^{\perp}.$ 
Hence in this case, we have $\eta = \mathcal{O}(\tau)$ and $\theta = \mathcal{O}(1)$, for $\tau\gg1$.
For the penetrable obstacle case, we choose $a$ and $b$ such that $a\perp\rho,a\perp\rho^{\perp}$ and $b=\bar{\hat\zeta}$, where $\hat\zeta = \frac{\zeta}{|\zeta|}$. 
In this case, we have $\eta = \mathcal{O}(1)$ and $\theta = \mathcal{O}(\tau)$, for $\tau\gg1$.
Adding a parameter $t>0$ into the CGO-solutions, we set, using the same notations,
\begin{equation}
\begin{cases}\label{CGO-form}
& E_0  := \eta e^{\{{\tau(x{\cdot}\rho-t)}+i\sqrt{{\tau}^2+k^2}{x\cdot{\rho}^\perp}\}}, \\ 
& H_0  := \theta e^{\{{\tau(x{\cdot}\rho-t)}+i\sqrt{{\tau}^2+k^2}{x\cdot{\rho}^\perp}\}}.
\end{cases}
\end{equation}
\textbf{Indicator Function:}
For $\rho\in\mathbb{S}^2$, $\tau>0 $ and $t>0$ we define the indicator function 
\begin{equation*}
I_\rho(\tau,t):= ik\tau\int_{\partial\Omega}(\nu\wedge E_0)\cdot(\overline{(\Lambda_D-\Lambda_\emptyset)(\nu\wedge E_0)\wedge\nu})dS
\end{equation*}
where $E_0$ is the CGO solution of Maxwell's equations given above. \\
\textbf{Support Function:}
For $\rho \in\mathbb{S}^2,$ we define the support function of $D$ by 
$h_D(\rho):=\sup_{x\in D} x \cdot\rho.$\\
Now, we state our main result.
\begin{theorem}\label{theorem1}
Let $\rho \in \mathbb{S}^{2}$.
 For both the penetrable and the impenetrable cases, we have the following characterizations of $h_{D}(\rho)$.
\begin{align}
& \vert I_{\rho}(\tau,t)\vert \leq Ce^{-c\tau},\; \tau >>1, \; c, C>0, \mbox{ and in particular, }  \lim_{\tau \to \infty}\vert I_{\rho}(\tau,t)\vert = 0 \quad (t>h_D(\rho))\label{main_behavior_1},\\
& \liminf_{\tau \to \infty}|I_{\rho}(\tau, h_{D}(\rho))|>0, \; \label{main_behavior_2} \\
&\vert I_{\rho}(\tau,t) \vert \geq Ce^{c\tau}, \; \tau >>1,\; c, C>0
, \mbox{ and in particular, } \lim_{\tau \to \infty}|I_{\rho}(\tau,t)| = \infty \quad
(t<h_D(\rho)) \label{main_behavior_3}.
\end{align}
\end{theorem}
From this theorem, we see that, for a fixed direction $\rho$, the behavior of the indicator
 function $I_{\rho}(\tau,t)$ changes drastically
in terms of $\tau$: exponentially decaying if $t>h_D(\rho)$, polynomially behaving if $t=h_D(\rho)$
and exponentially growing if $t<h_D(\rho)$. This feature can be used to reconstruct the support function 
$h_D(\rho),\; \rho \in \mathbb{S}^{2}$ from the data: $\Lambda _D (\nu \wedge E_0|_{\partial \Omega})$ with $E_0$ given by the CGO solutions. 
Hence, using Theorem \ref{theorem1}, we can reconstruct the convex hull of $D$. It is worth mentioning that using other CGO solutions, 
as it is proposed in \cite{Zhou}, we can reconstruct parts of the non-convex part of $D$. 

Before we discuss more Theorem \ref{theorem1}, let us recall that the idea of using CGO solutions for reconstructing interfaces goes back to 
\cite{MR1694840Ikehata_how} where the acoustic case has been considered, see also \cite{Ik-2002} and the references therein. We should also mention the works \cite{IISS}, \cite{NUW}, \cite{NY} and 
\cite{Y} where different CGOs were used for both the impenetrable and the penetrable obstacles in the acoustic case. Corresponding
results for the Lam{\'e} model with zero frequencies are given in \cite{Ik-It} and \cite{UW}.

Regarding the Maxwell case, Theorem \ref{theorem1} has been already proved for the impenetrable case in \cite{Zhou}. 
Our contribution in this paper is twofold.

The first contribution concerns the impenetrable case. We prove Theorem \ref{theorem1} with only Lipschitz regularity assumption on $\partial D$, 
but most importantly we impose no restriction on the directions $\rho \in \mathbb{S}^{2}$, while in \cite{Zhou} a countable set of 
such directions has to be avoided. This last restriction is  not natural and it is related to the geometrical assumption on the positivity of
lower bound of the curvature of the obstacle's surface. Such an assumption was already used in the previous works concerning the acoustic case, 
see \cite{MR1694840Ikehata_how} and \cite{NY} for instance. In \cite{SY}, see also \cite{Mourad1}, this assumption has been removed by proving a natural 
estimate of the so-called reflected solution. 
This estimate is obtained by using invertibility properties related to the layer potentials in the Sobolev spaces $H^s(\partial D)$, 
$s \in [-1, 1]$. 
In this current work, we generalize this technique to the Maxwell's model, in the corresponding spaces $X_{\partial D}^{-1/p,p}, p<2$, near $2$, see Section 2
for more details on these spaces, and we prove the needed estimate, see Proposition \ref{proposition1}, with
which  we can avoid the mentioned geometrical assumptions.

The second contribution of this paper is to justify Theorem \ref{theorem1} for the penetrable case. In this case also one needs an appropriate 
estimate of the corresponding reflected solution. In the acoustic case, see \cite{SY}, the analysis is based on the celebrated Meyers's 
$L^p$ estimate of the gradient of the solutions of scalar divergence form elliptic equations to provide a natural estimate of this reflected solution. 
Similar as in the impenetrable case, this estimate helps to avoid the apriori geometrical assumption of $\partial D$ used in the previous works, 
see for instance \cite{NUW}. Following this technique, we first prove a global $L^p$ estimate for the curl of the solutions of the Maxwell 
equations, for $p$ near 
$2$ and $p\leq 2$, in the spirit of Meyers's result, and then use it to provide the corresponding estimate which helps justifying 
Theorem \ref{theorem1} with no geometrical assumption and assuming minimum regularity on the interface $\partial D$ and the coefficient $\mu$.
We want to emphasize that this $L^p$ estimate is of importance for itself since it can be used for other purposes. A detailed discussion about this issue is
given in Section $4$. Let us mention here the $L^p$ regularity of the solution of the Maxwell's system, shown in 
\cite{BYZ}, see  also the references therein, where $\mu$ is taken as a constant and $\epsilon$ is piecewise constant. They derive an estimate of 
the magnetic field $H$ in the $W^{1, p}(\Omega)$-norm in terms of its $L^p(\Omega)$-norm and the data, for every $p$, $p\in (1, \infty)$. This means that the solution operator 
can be a bijection, however it is not necessary an isomorphism. The estimate we obtain here, where $\mu$ is taken in $L^{\infty}(\Omega)$ and 
$\epsilon$ constant, shows that this solution operator is an isomorphism on the spaces $H^{1, p}(\curl;\Omega)$ but for 
$p$ near $2$ and $p\leq 2$. 
This last property is important to our analysis of the enclosure method. Other regularity results can be found in \cite{Yin} 
where a global estimate in the
Campanato spaces are given and then a H{\"o}lder regularity estimate is shown. 

The paper is organized as follows. In Section $2$, we deal with the impenetrable case and in Section $3$, 
we consider the penetrable obstacle case. In Section $4$, we establish the global $L^p$ estimate for the curl of the solutions of 
the Maxwell's equations while in Section $5$, as an Appendix, we recall some important properties related to the Layer 
potentials and the Sobolev spaces appearing in the study of problems related to the Maxwell equations as well as 
the proof of some technical results used in Section $2$. 
\begin{section}{\textbf {{Proof of Theorem \ref{theorem1} for the impenetrable case}}}
We give the proof for the second point (\ref{main_behavior_2}), since it is the most
difficult part. The other points are easy to obtain by the identity $I_{\rho}(\tau,t) =
e^{2\tau(h_{D}(\rho)-t)} I_{\rho}(\tau,h_{D}(\rho))$ and (\ref{main_behavior_2}). In addition, the lower estimate in 
(\ref{main_behavior_2}) is the most difficult part since the upper bound is easy due to the well posed-ness of the forward problem.
 So, our focus is on the lower order estimate.
Let us recall the integration by parts formula from [\cite{Monk}, Theorem 3.29 and Theorem 3.31].
For any $v\in H(\curl;\Omega)$ and $\varphi\in(H^1(\Omega))^3$, the following Green's theorem holds.
\begin{equation}\label{intPART1}
\int_{\Omega}(\curl v)\cdot\varphi dx - \int_{\Omega}v\cdot(\curl\varphi)dx = \int_{\partial\Omega}(\nu\wedge v)\cdot\varphi ds(y). 
\end{equation}
In the other hand, for any $v\in H(\curl;\Omega)$ and $\varphi\in H(\curl;\Omega)$, we have
\begin{equation}\label{intPART2}
\int_{\Omega}(\curl v)\cdot\varphi dx - \int_{\Omega}v\cdot(\curl\varphi)dx = \int_{\partial\Omega}(\nu\wedge v)\cdot((\nu\wedge\varphi)\wedge\nu) ds(y). 
\end{equation}
We start by the following lemma.
\begin{lemma}\label{lemma3.1}
Assume $(E,H)\in{H(\curl;\Omega\setminus\overline{D})}\times H(\curl;\Omega\setminus\overline{D})$ is a solution of the problem
\begin{equation}\label{II}
\begin{cases}  
 &\curl E-ikH=0 \ \ \mbox{in}\ \Omega\setminus{\overline{D}},\\
  &\curl H+ikE=0 \ \ \text{in}\ \Omega\setminus{\overline{D}},\\
& \nu\wedge E=f \ \in{TH^{-\frac{1}{2}}(\partial\Omega)} \ \ \text{on}\ \partial\Omega,\\
& \nu\wedge H=0 \ \ \text{on}\ \partial{D},
\end{cases}
\end{equation}
with $f = {\nu}\wedge {{E_0}|_{\partial\Omega}}$. 
Then we have the identity,
\begin{equation*}
\begin{split}
-\frac{1}{\tau} I_{\rho}(\tau,t) 
& = -\int_{D}\{{\vert\curl E_0(x)\vert}^2-k^2{\vert Ẹ_0(x) \vert}^2\}dx - \int_{\Omega\setminus\overline{D}}\{{\vert\curl\tilde{E} (x)\vert}^2-{k^2}{\vert{\tilde{E}}(x)\vert}^2\}dx \\
& = \int_{D}\{|\curl H_0(x)|^2-k^2|H_0(x)|^2\}dx + \int_{\Omega\setminus\overline{D}}\{|\curl\tilde{H} (x)|^2-k^2|\tilde{H}(x)|^2\}dx
\end{split}
\end{equation*}
and then the inequality
\begin{equation}\label{3.2}
-\frac{1}{\tau} I_{\rho}(\tau,t) \geq \int_{D}\{|\curl H_0(x)|^2-k^2|H_0(x)|^2\}dx -k^2\int_{\Omega\setminus\overline{D}}|\tilde{H}(x)|^2dx,
\end{equation}
where $\tilde{E} := E - E_0 $ and $\tilde{H} := H - H_0 $.
\end{lemma}            
 \begin{proof}
 It is based on integration by parts, see Lemma 4.4 of \cite{Zhou} for details, keeping in mind that we use here the integration by parts formula 
\eqref{intPART2} since $E,H\in H(\curl;\Omega\setminus\overline{D})$.       
\end{proof}
Using \eqref{3.2}, we remark that it is enough to dominate the lower order term $\int_{\Omega\setminus\overline{D}}|\tilde{H}(x)|^2dx $
by the terms involving only $H_0$. Then from the explicit form of $H_0$, we deduce Theorem \ref{theorem1}.
This is the object of the next subsection. For this, we need the following extra condition on the wave number $k$. Namely, we assume that $k$
is not a Maxwell eigenvalue\footnote{This condition is needed because, for simplicity, we used the single layer potential representation \eqref{layer}.
It can be removed if we use combined single and double layer representations.} in $D$, i.e. if $E, H\in H(\curl;D)$ satisfy
\begin{equation*}
\begin{cases}  
 &\curl E-ikH=0 \ \ \mbox{in}\ D,\\
  &\curl H+ikE=0 \ \ \text{in}\ D,\\
& \nu\wedge E=0 \ \ \text{on}\ \partial{D},
\end{cases}
\end{equation*}
then $E=H=0$ in $D$.
\begin{subsection}{\textbf {Estimates of the lower order term $\tilde{H}$}} 
The aim is to prove the following estimate.
\end{subsection}
\begin{proposition}\label{proposition1}
Let $\Omega$ be $C^1$-smooth and D, $D\subset\Omega$, be Lipschitz. Then, there exists a positive constant C independent on
$(\tilde{E},\tilde{H})$ and $(E_0,H_0)$ such that
\begin{equation}\label{3.3}
\int_{\Omega\setminus\overline{D}}|\tilde H(x)|^2dx
  \leq C \{\|\curl H_0\|_{L^p(D)}^{2} + \|H_0\|_{H^{s+1/2}(D)}^{2}\},
 \end{equation}
for all $p$ and $s$ such that $\max\{2-\delta,4/3\}< p \leq 2$ and $0<s\leq1$\footnote{To extend this result to $s=0$, we need the trace theorem between $H^{1/2}(D)$ and $L^2(\partial D)$.
However, this trace theorem is not necessarily valid for Lipschitz domains, see for instance [\cite{Mclean}, p 209].} with $\delta>0$. 
\end{proposition}
\begin{proof}
\textbf{Step 1.}
 Let $E^{ex}, H^{ex}$ be the solution of the following well posed exterior problem, see \cite{Mitrea, Monk}.
\begin{equation}\label{exterior}
 \begin{cases}
 & \curl E^{ex}-ikH^{ex}=0  \ \ \text{in}\ \mathbb{R}^3\setminus\overline{D}, \\
 & \curl H^{ex} + ikE^{ex}=0  \ \ \text{in}\ \mathbb{R}^3\setminus\overline{D},\\
 & \nu\wedge H^{ex}=-\nu\wedge H_0 \ \ \text{on}\ \partial D,\\
 & E^{ex}, H^{ex} \ \text{satisfy the Silver-M{\"u}ller radiation condition.}
\end{cases}
\end{equation}
We represent these solutions $E^{ex}$ and $H^{ex}$ by the following layer potentials
\begin{equation}\label{layer}
 \begin{split}
 & H^{ex}(x) := \curl \int_{\partial D}\Phi_k(x,y)f(y)ds(y),\\
 & E^{ex}(x) := -\frac{1}{ik}\curl H^{ex}(x), \ \ x\in\mathbb{R}^3\setminus\partial D,
 \end{split}
\end{equation}
where $\Phi_k(x,y) := -\frac{e^{ik|x-y|}}{4\pi|x-y|},\ x,y\in\mathbb{R}^3, \ x\neq y,$ is the fundamental solution
of the Helmholtz equation and $f$ is the density. Note that
\eqref{layer} satisfy the first two equations and the radiation condition of \eqref{exterior}. By using the jump formula of the 
$\curl$ of the single layer potential on $\partial D$ with $X_{\partial D}^{-1/p,p}$ densities, see \cite{Mitrea}, 
where the space $X_{\partial D}^{-1/p,p}\subset L_{\tan}^{p}(\partial D)$ is defined as 
\begin{equation*}
X_{\partial D}^{-1/p,p}: = \{A\in L_{\tan}^{p}(\partial D);  {\Div A \in{W^{-1/p,p}(\partial{D})}}\}
\end{equation*} 
with the norm $\|A\|_{X_{\partial D}^{-1/p,p}} := \|A\|_{L^p(\partial D)} + \|\Div A\|_{W^{-1/p,p}(\partial{D})}$,
we obtain
\begin{equation}\label{jump}
 \nu\wedge H^{ex} = (-\frac{1}{2}I + M_k)f,
\end{equation}
where $M_k$ is defined as
\[
 (M_kf)(x) := \nu\wedge p.v. \curl \int_{\partial D}\Phi_k(x,y)f(y)ds(y), \ \ x\in\partial D.
\]
Hence $f$ is solution of the equation
\begin{equation}\label{buonjump}
 (-\frac{1}{2}I + M_k)f = -\nu\wedge H_0.
\end{equation}
We need the following lemma for our analysis.
\begin{lemma}[Theorem 5.3 of \cite{Mitrea}]\label{lemma3.4}
Let D be a bounded Lipschitz domain in $\mathbb {R}^3$ with ${\mathbb{R}^3}\setminus \overline{D}$ is connected. There exists $\delta$
positive and depending only on $\partial D$ such that, if $k \in{\mathbb{C}\setminus{\{0\}}},$ $Im\ k \geq 0,$ is not a Maxwell eigenvalue
for D, then the following operator is isomorphism
 \begin{equation*}
  (-{\frac{1}{2}}I + M_k) : X_{\partial D}^{-1/p,p}\longrightarrow X_{\partial D}^{-1/p,p}
 \end{equation*}
for each ${2-\delta}\leq p\leq{2+\delta}$.
\end{lemma}
Let us recall the Sobolev-Besov space $B_{\frac{1}{p}}^{p,2}(\Omega\setminus\overline{D})= [L^p(\Omega\setminus\overline{D}),W^{1,p}(\Omega\setminus\overline{D})]_{\frac{1}{p},2}$, 
see  Appendix A for a general setting and \cite{Mitrea, T.Muramatu} for more details. The embedding
$i : B_{1/p}^{p,2}(\Omega\setminus\overline{D}) \rightarrow L^2(\Omega\setminus\overline{D})$ is bounded for $4/3<p\leq2$, see for instance [\cite{T.Muramatu}, Theorem 2]. 
Using this embedding and Property \ref{pr7} in Appendix, we obtain
\begin{align}
 \|H^{ex}\|_{L^2(\Omega\setminus\overline{D})} \nonumber
& \leq C\|H^{ex}\|_{B_{1/p}^{p,2}(\Omega\setminus\overline{D})} \nonumber \\
& \leq C\{\|\nu\wedge H^{ex}\|_{L^p(\partial\Omega)} + \|\nu\wedge H^{ex}\|_{L^p(\partial D)} + \|H^{ex}\|_{L^p(\Omega\setminus\overline{D})} + \|E^{ex}\|_{L^p(\Omega\setminus\overline{D})}\}. \label{misti}
\end{align}
We denote the single layer potential by $\mathcal{S}_k$  
\[
 \mathcal{S}_kf(x) := \int_{\partial D}\Phi_k(x,y)f(y)ds(y), \ \ x\in\mathbb{R}^3\setminus\partial D.
\]
The operator $\mathcal{S}_k : W^{-1/p,p}(\partial D) \rightarrow W^{1,p}(\Omega\setminus\overline{D})$ 
is bounded, see Property \ref{pr5} in Appendix. Hence 
\begin{equation}\label{sadhona}
 \|H^{ex}\|_{L^p(\Omega\setminus\overline{D})}
= \|\curl\mathcal{S}_kf\|_{L^p(\Omega\setminus\overline{D})}
\leq C \|f\|_{W^{-1/p,p}(\partial D)} \leq C \|f\|_{L^p(\partial D)}.
\end{equation}
Now using the identity $\curl\curl\mathcal{S}_kf = \nabla\div\mathcal{S}_kf - \Delta\mathcal{S}_kf 
= \nabla\mathcal{S}_k(\Div f) + k^2\mathcal{S}_kf$
and the above properties of the single layer potential, we obtain
\begin{align}
 \|E^{ex}\|_{L^p(\Omega\setminus\overline{D})} \nonumber
& \leq C \|\curl\curl\mathcal{S}_kf\|_{L^p(\Omega\setminus\overline{D})} \nonumber \\
& \leq C [\|\nabla\mathcal{S}_k(\Div f)\|_{L^p(\Omega\setminus\overline{D})} + \|\mathcal{S}_kf\|_{L^p(\Omega\setminus\overline{D})}] \nonumber \\
& \leq C [\|\Div f\|_{W^{-1/p,p}(\partial D)} + \|f\|_{W^{-1/p,p}(\partial D)}] \nonumber \\
& \leq C [\|\Div f\|_{W^{-1/p,p}(\partial D)} + \|f\|_{L^p(\partial D)}]. \label{bujhi}
\end{align}
Also, since $\overline{D}\subset\subset\Omega$ we have
\begin{align}
 \|\nu\wedge H^{ex}\|_{L^p(\partial\Omega)} \nonumber 
& \leq C [\int_{\partial\Omega}(\int_{\partial D}\curl_x\Phi_k(x,y)f(y)ds(y))^pds(x)]^{1/p} \nonumber \\
& \leq C \|f\|_{L^p(\partial D)}. \label{europe}
\end{align}
Combining the estimates \eqref{misti}, \eqref{sadhona}, \eqref{bujhi}, \eqref{europe} and the fact that $\nu\wedge H^{ex} = -\nu\wedge H_0$ on $\partial D$,
we have
\begin{equation}\label{exTR}
 \|H^{ex}\|_{L^2(\Omega\setminus\overline{D})}
\leq C\|\nu\wedge H_0\|_{L^p(\partial D)} + C[\|f\|_{L^p(\partial D)} + \|\Div f\|_{W^{-1/p,p}(\partial D)}].
\end{equation}
Using the invertibility of the operator $-\frac{1}{2}I + M_k$, see Lemma \ref{lemma3.4}, from the equation \eqref{buonjump},
we obtain
\begin{equation}\label{samsung}
 \|f\|_{L^p(\partial D)} + \|\Div f\|_{W^{-1/p,p}(\partial D)} 
\leq C [\|\nu\wedge H_0\|_{L^p(\partial D)} + \|\Div(\nu\wedge H_0)\|_{W^{-1/p,p}(\partial D)}].
\end{equation}
Hence from Property \ref{pr4} and Property \ref{pr3} of Theorem 5.1 in Appendix together with the estimates \eqref{exTR} and \eqref{samsung}
we obtain
\begin{equation}\label{exteriIIIo}
 \|H^{ex}\|_{L^2(\Omega\setminus\overline{D})} \leq C[\|\nu\wedge H_0\|_{L^p(\partial D)} + \|\curl H_0\|_{L^p(D)}], \ \ 4/3<p\leq 2.
\end{equation}
\textbf{Step 2.}
Define, $\mathcal{E} := \tilde E - E^{ex}$ and $\mathcal{H} := \tilde H - H^{ex}$, then $\mathcal{E}$ and $\mathcal{H}$ satisfy the following Maxwell problem
\begin{equation}\label{mcal}
 \begin{cases}
 & \curl \mathcal{E}-ik\mathcal{H}=0  \ \ \text{in}\ \Omega\setminus\overline{D}, \\
 & \curl \mathcal{H} + ik\mathcal{E}=0  \ \ \text{in}\ \Omega\setminus\overline{D},\\
 & \nu\wedge \mathcal{H}= 0 \ \ \text{on}\ \partial D,\\
 & \nu\wedge \mathcal{E}= -\nu\wedge E^{ex} \ \ \text{on}\ \partial\Omega.\\
\end{cases}
\end{equation}
Applying the $L^2$-theory for the Maxwell system, we obtain
 \begin{align}
  \|\mathcal{H}\|_{L^2(\Omega\setminus\overline{D})}
\leq \|\mathcal{E}\|_{H(\curl;\Omega\setminus\overline{D})} 
 \leq C \|\nu\wedge\mathcal{E}\|_{H^{-1/2}(\partial\Omega)}  
 \leq C \|\nu\wedge E^{ex}\|_{H^{-1/2}(\partial\Omega)}. \label{2ndMaxmm}
\end{align}
For $x\in\partial\Omega$ and $y\in\partial D$, the fundamental solution $\Phi_k(x,y)$ is a smooth function. Therefore from \eqref{layer}, we have
\begin{align}
|_{H^{-1/2}(\partial\Omega)}\langle\nu\wedge E^{ex},\varphi\rangle_{H^{1/2}(\partial\Omega)}|
& = |\int_{\partial\Omega}(\nu(x)\wedge E^{ex}(x))\varphi(x)ds(x)|\nonumber \\
& = |\int_{\partial\Omega}\int_{\partial D}(\nu(x)\wedge\curl_x\curl_x\Phi_k(x,y))f(y)\varphi(x)ds(x)ds(y)|\nonumber \\
&\leq \int_{\partial\Omega}\int_{\partial D}|\nu(x)\wedge\curl_x\curl_x\Phi_k(x,y)||f(y)||\varphi(x)|ds(x)ds(y) \nonumber \\
&\leq C \left(\int_{\partial D}|f(y)|ds(y)\right)\left(\int_{\partial\Omega}|\varphi(x)|ds(x)\right) \nonumber \\
& \leq C \|f\|_{L^p(\partial D)}\|\varphi\|_{H^{1/2}(\partial\Omega)}. \nonumber
\end{align}
Taking the supremum over all $\varphi$ with $\|\varphi\|_{H^{1/2}(\partial\Omega)}\leq 1$ on the above estimate, we get
\begin{equation}\label{extof}
\|\nu\wedge E^{ex}\|_{H^{-1/2}(\partial\Omega)} \leq C \|f\|_{L^p(\partial D)}, \ \ \forall p\geq1 .
\end{equation}
From \eqref{2ndMaxmm} and \eqref{extof} together with Lemma \ref{lemma3.4} and \eqref{buonjump}, we obtain 
\begin{align}
\|\mathcal{H}\|_{L^2(\Omega\setminus\overline{D})} 
& \leq C [\|f\|_{L^p(\partial D)} + \|\Div f\|_{W^{-1/p,p}(\partial D)}]  \nonumber \\
& \leq C [\|\nu\wedge H_0\|_{L^p(\partial D)} + \|\Div(\nu\wedge H_0)\|_{W^{-1/p,p}(\partial D)}] \nonumber \\ 
& \leq C [\|\nu\wedge H_0\|_{L^p(\partial D)} + \|\curl H_0\|_{L^p(D)}], \ \ 2-\delta\leq p\leq2+\delta. \label{maxcall}
 \end{align}
Combining \eqref{exteriIIIo} and \eqref{maxcall}, we obtain
\begin{equation}\label{defenCCE}
 \begin{split}
  \int_{\Omega\setminus\overline{D}}|\tilde H(x)|^2dx
& \leq \|\mathcal{H}\|_{L^2(\Omega\setminus\overline{D})}^{2} + \|H^{ex}\|_{L^2(\Omega\setminus\overline{D})}^{2} \\
& \leq C [\|\nu\wedge H_0\|_{L^p(\partial D)}^{2} + \|\curl H_0\|_{L^p(D)}^{2}], 
 \end{split}
\end{equation}
for all $\max\{2-\delta,4/3\}< p\leq2.$ 
As, for $s>0$ and $p\leq2$ we have $H^s(\partial D) \subset L^2(\partial D)\subset L^p(\partial D)$, then we deduce that 
\[
 \|\nu\wedge H_0\|_{L^p(\partial D)} \leq C \|H_0\|_{L^p(\partial D)} \leq C \|H_0\|_{H^s(\partial D)}.
\]
Note that the trace map $\gamma : H^{s+1/2}(D) \rightarrow H^s(\partial D)$, $0<s\leq 1$\footnote{This is the place where we used the trace theorem for Lipschitz domain and 
we need to avoid $s=0$.}, defined by $\gamma(u) = u|_{\partial D}$, 
is bounded. So the estimate \eqref{defenCCE} becomes
\begin{equation}
 \int_{\Omega\setminus\overline{D}}|\tilde H(x)|^2dx \leq C [\|H_0\|_{H^{s+1/2}(D)}^{2} + \|\curl H_0\|_{L^p(D)}^{2}],
\end{equation}
for all $\max\{2-\delta,4/3\}< p\leq2 $ with $\delta>0$ and $0<s\leq1.$ 
\end{proof}
\begin{subsection}{\textbf{Proof of Theorem \ref{theorem1}}}
 \end{subsection}     
      Here, we use the same notations as in the previous works \cite{MR1694840Ikehata_how}, \cite{NUW} and \cite{NY} for instance. Let us first introduce the sets $D_{j,\delta}\subset D, D_{\delta}\subset D$ as follows. 
For any $\alpha \in {\partial D \cap\{x\cdot\rho = h_{D}(\rho)\}} =: K,$
we define $ B(\alpha,\delta) := \{x\in \mathbb{R}^3 ; | x- \alpha | < \delta\} (\delta > 0).$
Then, $K \subset \cup_{\alpha\in K}B(\alpha ,\delta).$ Since K is compact, 
there exist $\alpha_1,\cdots,\alpha_n$ such that $ K \subset {B(\alpha_{1},\delta)\cup \cdots \cup B(\alpha_{n},\delta)}$. 
Then we define 
$
 D_{j,\delta}:= D \cap B(\alpha_{j},\delta), D_{\delta}:= \cup_{j=1}^{n}D_{j,\delta}.
$
Also
\[
 \int_{D \setminus D_ \delta} {e^{-p \tau (h_D(\rho)-x \cdot \rho)}} dx = O(e^{-pc \tau}) \ \ (\tau \rightarrow \infty ),
\]
with positive constant $c$.
Let $\alpha_j \in K. $ By a rotation and translation, we may assume that $\alpha_j = 0$ and the vector $\alpha_j - x_0 = 0$ is
parallel to $e_3 = (0,0,1).$ Then, we consider a change of coordinate near $\alpha_j :$
\begin{equation}\label{3.27}
 y' = x', y_3 = h_D(\rho) - x \cdot \rho,
\end{equation}
where $x' = (x_1,x_2), y' = (y_1,y_2), x = (x',x_3), y= (y',y_3).$ Denote the parametrization of $\partial D$ near $\alpha_j$
by $l_j(y').$
\begin{lemma}\label{lemma3.9}
For $ 1 \leq q < \infty $, The following estimates hold. 
\bigskip

\noindent 1.  
\begin{equation}\label{3.23}
 \int_D {| {H_0(x)} |}^qdx \leq C{\tau}^{-1} \sum_{j=1}^{n}\iint_{{| y'|} < \delta} e^{-q\tau l_j(y')}dy' - \frac{C}{q}\tau^{-1} e^{-q \delta \tau}
+ C e^{-qc \tau}
\end{equation}
\bigskip
\noindent 2.
\begin{equation}\label{3.24}
 \int_D {| {H_0(x)} |}^2dx \geq C{\tau}^{-1} \sum_{j=1}^{n}\iint_{{| y'|} < \delta} e^{-2\tau l_j(y')}dy' - \frac{C}{2}\tau^{-1} e^{-2 \delta \tau}
\end{equation}
\bigskip
\noindent 3.
\begin{equation}\label{3.25}
 \int_D {| {\curl H_0(x)} |}^qdx \leq C{\tau}^{q-1} \sum_{j=1}^{n}\iint_{{| y'|} < \delta} e^{-q\tau l_j(y')}dy' - \frac{C}{q}\tau^{q-1} e^{-q \delta \tau}
+ C \tau^q e^{-qc \tau}
\end{equation}
\bigskip
\noindent 4.
\begin{equation}\label{3.26}
 \int_D {| {\curl H_0(x)} |}^2dx \geq C{\tau} \sum_{j=1}^{n}\iint_{{| y'|} < \delta} e^{-2\tau l_j(y')}dy' - \frac{C}{2}\tau e^{-2 \delta \tau}.
\end{equation}
\end{lemma}
\begin{proof}
 We only give the proofs for the points \noindent 1. and \noindent 2.
Recall that, for $t > 0$ we are considering the CGO solutions as follows.
\begin{equation}
\begin{cases}\label{CGO-formImpeneq}
& E_0  := \eta e^{\{{\tau(x{\cdot}\rho-t)}+i\sqrt{{\tau}^2+k^2}{x\cdot{\rho}^\perp}\}}, \\ 
& H_0  := \theta e^{\{{\tau(x{\cdot}\rho-t)}+i\sqrt{{\tau}^2+k^2}{x\cdot{\rho}^\perp}\}},
\end{cases}
\end{equation}
where $\eta = \mathcal{O}(\tau)$ and $\theta = \mathcal{O}(1)$, for $ \tau >>1$.
\bigskip

\noindent 1.
\begin{align}
 \int_D |H_0(x)|^q dx 
& = \int_D e^{-q \tau(h_D(\rho)-x \cdot \rho )}|\theta|^q dx \nonumber \\
& \leq C \int_D {e^{-q \tau(h_D(\rho)-x \cdot \rho )}} dx \nonumber \\
& = \int_{D_\delta} {e^{-q \tau(h_D(\rho)-x \cdot \rho )}} dx + \int_{D\setminus D_\delta} {e^{-q \tau(h_D(\rho)-x \cdot \rho )}} dx \nonumber \\
& \leq C  \sum_{j=1}^{n}\iint_{|y'| < \delta} dy' \int_{l_j(y')}^{\delta} e^{-q \tau y_3} dy_3 + C e^{-qc \tau} \nonumber \\
& = C{\tau}^{-1} \sum_{j=1}^{n}\iint_{|y'| < \delta} e^{-q\tau l_j(y')}dy' - \frac{C}{q}\tau^{-1} e^{-q \delta \tau}
  + C e^{-qc \tau}. \nonumber
\end{align}
\bigskip
\noindent 2.
 \begin{align*}
 \int_D |H_0(x)|^2 dx 
&= \int_D e^{-2 \tau(h_D(\rho)-x \cdot \rho )}|\theta|^2 dx \\
& \geq C \int_D {e^{-2 \tau(h_D(\rho)-x \cdot \rho )}} dx 
 \geq C \int_{D_\delta} {e^{-2 \tau(h_D(\rho)-x \cdot \rho )}} dx \\ 
& \geq C  \sum_{j=1}^{n}\iint_{|y'| < \delta} dy' \int_{l_j(y')}^{\delta} e^{-2 \tau y_3} dy_3  \\
& = C{\tau}^{-1} \sum_{j=1}^{n}\iint_{|y'| < \delta} e^{-2\tau l_j(y')}dy' - \frac{C}{2}\tau^{-1} e^{-2 \delta \tau}. 
\end{align*}       
\end{proof}
\begin{lemma}\label{lemma3.10}
We have the following estimate
\begin{equation}\label{3.28}
 \frac{\|H_0\|_{L^2(D)}^{2}}{\|\curl H_0\|_{L^2(D)}^{2}} \leq \mathcal{O}(\tau^{-2}), \ \ \tau\gg1. 
\end{equation}
\end{lemma}
\begin{proof}
We have the following estimate
\begin{align}
\sum_{j=1}^n \iint_{|y'|<\delta} e^{-2\tau l_j(y')} dy' 
&\geq C\sum_{j=1}^n \iint_{|y'|<\delta} e^{-2\tau |y'|} dy' \nonumber \\
&\geq C\tau^{-2} \sum_{j=1}^n \iint_{|y'|<\tau \delta} e^{-2 |y'|} dy' 
 =\mathcal{O}(\tau^{-2}), \label{lower_v}
\end{align}
 since we have $ l_j(y') \leq C | y'|$ if $ \partial D$ is Lipschitz. Now using Lemma \ref{lemma3.9} we obtain
\[
\begin{split}
 \frac{\|\curl H_0\|_{L^2(D)}^{2}}{\|H_0\|_{L^2(D)}^{2}} 
& \geq \frac{C\tau \sum_{j=1}^{n}\iint_{{| y'|}<\delta}e^{-2\tau l_j(y')}dy' - \frac{C}{2}\tau e^{-2\delta \tau}}{C\tau^{-1} \sum_{j=1}^{n}\iint_{{| y'|}<\delta}e^{-2\tau l_j(y')}dy' - \frac{C}{2}\tau^{-1} e^{-2\delta \tau} + Ce^{-2c \tau}} \\
&= C \tau^2 \frac{1-\frac{C}{2} \frac{e^{-2 \delta \tau}}{\sum_{j=1}^{n}\iint_{| y'| < \delta}e^{-2\tau l_j(y')}dy'}}{1-\frac{C}{2} \frac{e^{-2 \delta \tau}}{\sum_{j=1}^{n}\iint_{| y'| < \delta}e^{-2\tau l_j(y')}dy'} + \frac{C \tau e^{-2c\tau}}{\sum_{j=1}^{n}\iint_{| y'| < \delta}e^{-2\tau l_j(y')}dy'}} \\
& = \mathcal{O} (\tau ^2) \ \ (\tau\gg1). 
\end{split}
\]
\end{proof}
\begin{lemma}\label{neiii}
 For $p<2$, we have the following estimate
 \[
  \frac{\|\curl H_0\|_{L^p(D)}^{2}}{\|\curl H_0\|_{L^2(D)}^{2}} \leq C \tau^{1-\frac{2}{p}}, \ \ \tau\gg1.
 \]
\end{lemma}
\begin{proof}
Using the H{\"o}lder inequality with exponent $q=\frac{2}{p}>1$, we have:
$$\sum_{j=1}^n \iint_{|y'|<\delta} e^{-p\tau l_j(y')} dy'\leq C (\sum_{j=1}^n \iint_{|y'|<\delta} e^{-2\tau l_j(y')} dy')^{\frac{p}{2}}.$$
Using Lemma \ref{lemma3.9} and \eqref{lower_v}, we obtain
 \begin{align}
 \frac{\|\curl H_0\|_{L^p(D)}^{2}}{\|\curl H_0\|_{L^2(D)}^{2}} 
& = \frac{\left( \int_D |\curl H_0(x)|^p dx \right) ^{\frac{2}{p}}}{\int_D |\curl H_0(x)|^2dx} \nonumber \\
&\leq C\frac{\tau^{(p-1)\frac{2}{p}} \left[\left(\sum_{j=1}^n \iint_{|y'|<\delta} e^{-p\tau l_j(y')} dy'\right)^{\frac{2}{p}} +\mathcal{O}(\tau^{\frac{2}{p}}e^{-2\delta \tau}) + \mathcal{O}(e^{-2c\tau})\right]}{\tau \sum_{j=1}^n \iint_{|y'|<\delta} e^{-2\tau l_j(y')} dy' +\mathcal{O}(\tau e^{-2\delta \tau})} \nonumber \\
&\leq \tau^{1-\frac{2}{p}}\frac{ \sum_{j=1}^n \iint_{|y'|<\delta} e^{-2\tau l_j(y')} dy' +\mathcal{O}(e^{-2\delta \tau}) + \mathcal{O}(\tau^{\frac{2}{p}}e^{-2c\tau})}{ \sum_{j=1}^n \iint_{|y'|<\delta} e^{-2\tau l_j(y')} dy' +\mathcal{O}(e^{-2\delta \tau})} \nonumber \\
&=\tau^{1-\frac{2}{p}}\frac{ 1 +\frac{\mathcal{O}(e^{-2\delta \tau}) + \mathcal{O}(\tau^{\frac{2}{p}}e^{-2c\tau})}{\sum_{j=1}^N \iint_{|y'|<\delta} e^{-2\tau l_j(y')} dy'}}{ 1+\frac{\mathcal{O}(e^{-2\delta \tau})}{\sum_{j=1}^N \iint_{|y'|<\delta} e^{-2\tau l_j(y')} dy'}} \nonumber \\
&\leq C\tau^{1-\frac{2}{p}} \ \ (\tau\gg1). \label{pene_estimate}
  \end{align}
\end{proof}
\begin{lemma}\label{lemma3.11}
If $t = h_D(\rho),$ then for some positive constant C,
\begin{equation*}
 \liminf_{\tau\rightarrow \infty} \int_{D}\tau |\curl H_0(x)|^2 dx \geq C.
\end{equation*}
\end{lemma}
\begin{proof}
\begin{align*}
\int_{D}|\curl H_0(x)|^2 dx 
& \geq C \int_{D}|E_0(x)|^2 dx \\
& \geq C \tau\sum_{j=1}^{n}\iint_{|y'| < \delta} e^{-2\tau l_j(y')}dy' - \frac{C}{2}\tau e^{-2\delta \tau} \\
& \geq C\tau \sum_{j=1}^{n}\iint_{|y'| < \delta} e^{-2\tau |y'|}dy' - \frac{C}{2}\tau e^{-2\delta \tau} \\
& \geq C\tau \left[{\tau}^{-2}\sum_{j=1}^{n}\iint_{|y'| < \tau \delta} e^{-2|y'|}dy'\right] - \mathcal{O}(\tau e^{-2\delta \tau})\\
& \geq C\tau^{-1}, \ \ (\tau\gg1).
\end{align*}
Hence 
\begin{equation*}
 \liminf_{\tau\rightarrow \infty}  \int_{D}\tau|\curl H_0(x)|^2 dx \geq C>0. \ \ 
\end{equation*}
\end{proof}
\textbf{End of the proof of Theorem \ref{theorem1}}\\
Recall that from Lemma \ref{lemma3.1}, we have
\begin{equation*}
 -\frac{1}{\tau} I_{\rho}(\tau,t) \geq \int_{D}\{|\curl H_0(x)|^2-k^2|H_0(x)|^2\}dx -k^2\int_{\Omega\setminus\overline{D}}|\tilde{H}(x)|^2dx.
\end{equation*}
Now, from Proposition \ref{proposition1}, we deduce
\begin{equation}\label{MyGOd}
 -\frac{1}{\tau}I_{\rho}(\tau,h_D(\rho)) \geq \int_{D}\{|\curl H_0(x)|^2-k^2|H_0(x)|^2\}dx - C[\|H_0\|_{H^{s+1/2}(D)}^{2} + \|E_0\|_{L^p(D)}^{2}],
\end{equation}
where $0<s\leq 1$ and $4/3<p<2$.
We now estimate the term $\frac{\|H_0\|_{H^{s+1/2}(D)}^{2}}{\|\curl H_0\|_{L^2(D)}^{2}}$, for $0<s\leq 1$.
Set $t=s+1/2.$ Then we need to estimate $\frac{\|H_0\|_{H^t(D)}^{2}}{\|\curl H_0\|_{L^2(D)}^{2}},$ for $t\in(\frac{1}{2},\frac{3}{2}]$.
Using the interpolation inequality, we have 
\[
 \|H_0\|_{H^t(D)} \leq C \|H_0\|_{L^2(D)}^{1-t}\|H_0\|_{H^1(D)}^{t}, \ \ 0\leq t\leq 1.
\]
By the Young inequality 
$ab\leq \delta^{-\alpha}\frac{a^{\alpha}}{\alpha} + \delta^{\beta}\frac{b^{\beta}}{\beta}, \ \ \frac{1}{\alpha} + \frac{1}{\beta} = 1,$
we can write
\begin{equation}\label{interPOL}
 \|H_0\|_{H^t(D)}^{2} \leq C \left[\frac{\delta^{-\alpha}}{\alpha}\|H_0\|_{L^2(D)}^{2(1-t)\alpha} + \frac{\delta^{\beta}}{\beta}\|H_0\|_{H^1(D)}^{2t\beta} \right].
\end{equation}
Choose $\beta = t^{-1}, 0<t<1$. Hence, $\alpha = (1-t)^{-1}, 0<t<1.$
So, $2(1-t)\alpha=2$ and $2t\beta =2$. Then the estimate \eqref{interPOL} becomes
\begin{align}
 \|H_0\|_{H^t(D)}^{2}
 & \leq C \left[\frac{\delta^{-\alpha}}{\alpha}\|H_0\|_{L^2(D)}^{2} + \frac{\delta^{\beta}}{\beta}\|H_0\|_{H^1(D)}^{2}\right] \nonumber \\
 & \leq C \left[\{(1-t)\delta^{-(1-t)^{-1}}+t\delta^{t^{-1}}\}\|H_0\|_{L^2(D)}^{2}+ t\delta^{t^{-1}}\|\nabla H_0\|_{L^2(D)}^{2} \right]. \label{inTERRRR}
\end{align}
Recall that $H_0 = \theta e^{\tau(x\cdot\rho-t)+i\sqrt{\tau^2+k^2}x\cdot\rho^{\bot}},$ where $\theta = \mathcal{O}(1), \ \ \tau\gg1.$
Therefore, 
\[
\frac{\partial H_0}{\partial x_j} = \theta (\tau\rho_j+i\sqrt{\tau^2+k^2}\rho_{j}^{\bot})e^{\tau(x\cdot\rho-t)+i\sqrt{\tau^2+k^2}x\cdot\rho^{\bot}}.
\]
Hence
\begin{align*}
 \|\nabla H_0\|_{L^2(D)}^{2}
& = \sum_{j=1}^{3}\|\frac{\partial H_0}{\partial x_j}\|_{L^2(D)}^{2}\\
& \leq C \sum_{j=1}^{3}\int_D [\tau^2\rho_{j}^{2}+(\tau^2 +k^2){\rho_{j}^{\bot}}^2] e^{2\tau(x\cdot\rho-t)}dx\\
& \leq C \tau^2\int_D e^{2\tau(x\cdot\rho-t)}dx.
\end{align*}
For $t=h_D(\rho),$ we obtain
 \begin{align}
   \|\nabla H_0\|_{L^2(D)}^{2}
& \leq C \tau^2\int_D e^{-2\tau(h_D(\rho)-x\cdot\rho)}dx \nonumber \\
& = C\tau^2\left(\int_{D_{\delta}}+\int_{D\setminus \overline{D_{\delta}}}\right)e^{-2\tau(h_D(\rho)-x\cdot\rho)}dx \nonumber \\
& \leq C\tau^2 \sum_{j=1}^{n}\iint_{{| y'|} < \delta}dy'\int_{l_j(y')}^{\delta}e^{-2\tau y_3}dy_3 + C\tau^2e^{-2c\tau} \nonumber \\
& \leq C\tau \sum_{j=1}^{n}\iint_{|y'| < \delta}e^{-2\tau l_j(y')}dy' - \frac{C}{2}\tau e^{-2\delta\tau} + C \tau^2 e^{-2c\tau}. \label{KOREA}
 \end{align}
From Lemma \ref{lemma3.9} and \eqref{KOREA}, we have 
\begin{equation}\label{fractt}
 \frac{\|\nabla H_0\|_{L^2(D)}^{2}}{\|\curl H_0\|_{L^2(D)}^{2}} \leq C.
\end{equation}
Hence from \eqref{3.28} together with \eqref{inTERRRR} and \eqref{fractt} we obtain
\begin{align}
 \frac{\|H_0\|_{H^t(D)}^{2}}{\|\curl H_0\|_{L^2(D)}^{2}}
& \leq C \{(1-t)\delta^{-(1-t)^{-1}}+t\delta^{t^{-1}}\}\frac{\|H_0\|_{L^2(D)}^{2}}{\|\curl H_0\|_{L^2(D)}^{2}} + t\delta^{t^{-1}}\frac{\|\nabla H_0\|_{L^2(D)}^{2}}{\|\curl H_0\|_{L^2(D)}^{2}} \nonumber \\
& \leq C \{(1-t)\delta^{-(1-t)^{-1}}+t\delta^{t^{-1}}\}\mathcal{O}(\tau^{-2}) + Ct\delta^{t^{-1}}. \label{lions}
\end{align}
We now choose $p$ such that $\max\{2-\delta,4/3\}<p<2.$
Combining \eqref{MyGOd} and \eqref{lions} together with Lemma \ref{lemma3.10} and Lemma \ref{neiii}, we obtain
\begin{align}
 \frac{-\frac{1}{\tau}I_{\rho}(\tau,h_D(\rho))}{\|\curl H_0\|_{L^2(D)}^{2}}
& \geq C-c_1\frac{\|H_0\|_{L^2(D)}^{2}}{\|\curl H_0\|_{L^2(D)}^{2}}-c_2\frac{\|H_0\|_{H^t(D)}^{2}}{\|\curl H_0\|_{L^2(D)}^{2}}-c_3\frac{\|\curl H_0\|_{L^p(D)}^{2}}{\|\curl H_0\|_{L^2(D)}^{2}} \nonumber \\
& \geq C - c_1\{(1-t)\delta^{-(1-t)^{-1}}+t\delta^{t^{-1}}-1\}\mathcal{O}(\tau^{-2})-c_2t\delta^{t^{-1}}-c_3\tau^{1-\frac{2}{p}} \nonumber \\
& \geq C - c_2t\delta^{t^{-1}}, \ \ 1/2<t<1, \ \ \tau\gg1.      \label{amaRRRb}
\end{align}
Now, we fix $t$ in $(1/2,1)$ and choose $\delta>0$ such that $C - c_2t\delta^{t^{-1}}>c>0,$
then the estimate \eqref{amaRRRb} becomes
\[
 \frac{-\frac{1}{\tau}I_{\rho}(\tau,h_D(\rho))}{\|\curl H_0\|_{L^2(D)}^{2}} \geq c>0, \ \ \tau\gg1.
\]
Hence form Lemma \ref{lemma3.11} we obtain
\[
 \liminf_{\tau{\rightarrow}\infty} \vert{I_{\rho}(\tau,h_D(\rho))}\vert  \geq c>  0. 
\]
\end{section}
\begin{section}{\textbf {{Proof of Theorem \ref{theorem1} for the penetrable case}}}
In this section, we prove our main theorem for the penetrable obstacle case.
 For a wave number $k>0,$ electric permittivity $\epsilon >0 $ and magnetic permeability $\mu >0$, consider the penetrable obstacle problem as follows 
\begin{equation}\label{III}
\begin{cases} 
& \curl E-ik \mu H=0 \ \ \text{in}\ \Omega,\\
& \curl H+ik \epsilon E=0 \ \ \text{in}\ \Omega,\\
& \nu\wedge E=f \ \ \text{on}\ \partial\Omega
\end{cases}
\end{equation}
where $k$ is not an eigenvalue of the spectral problem corresponding to (\ref{III}). 
Recall that, in this section we use the CGO solutions of the form 
\begin{equation}
\begin{cases}\label{CGO-formPerTT}
& E_0  := \eta e^{\{{\tau(x{\cdot}\rho-t)}+i\sqrt{{\tau}^2+k^2}{x\cdot{\rho}^\perp}\}}, \\ 
& H_0  := \theta e^{\{{\tau(x{\cdot}\rho-t)}+i\sqrt{{\tau}^2+k^2}{x\cdot{\rho}^\perp}\}}
\end{cases}
\end{equation}
where $\eta = \mathcal{O}(1)$ and $\theta = \mathcal{O}(\tau)$ for all $\tau\gg1$ and $t>0.$
Let $\tilde E = E-E_0$ be the reflected solution. It satisfies the following problem
\begin{equation}\label{IV}
\begin{cases}
& \curl\frac{1}{\mu(x)}\curl\tilde E-k^2 \epsilon(x) \tilde E = -\curl(\frac{1}{\mu(x)}-1)\curl E_0 + k^2 (\epsilon(x)- 1)E_0 \ \ \text{in}\ \Omega,\\
& \nu\wedge \tilde E=0 \ \ \text{on}\ \partial\Omega.
\end{cases}
\end{equation}
\begin{lemma}\label{lemma4.1}
We have the estimates
\begin{equation*}
-\tau^{-1}I_{\rho}(\tau,t) \geq \int_{D}(1-\mu(x))|\curl E_0(x)|^2dx - k^2 \int_{\Omega} \vert \tilde E(x)\vert^2dx -k^2 \int_{D}(\epsilon(x)-1)\vert E(x)\vert^2 dx, 
\end{equation*}
and
\begin{equation*}
 \tau^{-1}I_{\rho}(\tau,t) \geq \int_{D}(1-\frac{1}{\mu(x)})\vert \curl E_0(x)\vert^2dx - k^2\int_{\Omega}\epsilon(x)\vert \tilde E(x)\vert^2dx + k^2\int_D(\epsilon(x)-1)\vert E_0(x)\vert^2dx.  
\end{equation*}
The first inequality will be used if $1-\mu(x)>0$ and the second one if $1-\mu(x)<0.$
\end{lemma}
\begin{proof}
\textbf{Step 1}
First we need to prove the following identity
\begin{align}
& -k^2 \int_{\Omega}(\epsilon(x)-1)| E_0(x)|^2dx + \int_{\Omega}(\frac{1}{\mu(x)}-1)|\curl E_0(x)|^2dx \nonumber \\
& + k^2 \int_{\Omega}\epsilon(x) |\tilde E(x)|^2dx - \int_{\Omega}\frac{1}{\mu(x)}|\curl\tilde E(x)|^2dx \nonumber \\
& = - \tau^{-1}I_\rho(\tau,t). \label{4.1}
\end{align}
Multiplying by $\overline{\tilde E}(x)$ in the equation \eqref{IV} and using integration by parts we obtain
\begin{equation*}
\int_{\Omega}\frac{1}{\mu(x)}{\vert \curl\tilde E(x)\vert}^2dx + \int_{\Omega}((\frac{1}{\mu(x)}-1)\curl E_0(x))\cdot (\curl\overline{\tilde E}(x))dx
-k^2 \int_{\Omega} \epsilon(x){\vert {\tilde E}(x) \vert}^2dx 
\end{equation*}
\begin{equation*}
 - k^2 \int_{\Omega}(\epsilon(x)-1)E_0(x) \cdot \overline{\tilde E}(x) dx = 0,
\end{equation*}
\begin{align}
& \int_{\Omega}\frac{1}{\mu(x)}{\vert \curl\tilde E(x) \vert}^2dx - \int_{\Omega} (\frac{1}{\mu(x)}-1){\vert \curl E_0(x)\vert}^2dx \nonumber \\
& - k^2 \int_{\Omega}\epsilon(x){\vert \tilde E(x)\vert}^2dx 
+ k^2\int_{\Omega}(\epsilon(x)-1){\vert E_0(x)\vert}^2 dx \nonumber \\
&= k^2\int_{\Omega}(\epsilon(x)-1)E_0(x) \cdot \overline{E}(x)dx 
- \int_{\Omega}(\frac{1}{\mu(x)}-1)(\curl E_0(x))\cdot (\curl\overline{E}(x))dx. \label{4.2}
\end{align}
On the other hand from equation \eqref{III} eliminating $H(x)$ we have  
\begin{equation}\label{Max_ep}
 \curl (\frac{1}{\mu(x)}\curl E(x)) - k^2 \epsilon E(x) = 0.
\end{equation}
Then multiplying by $E_0(x)$ in equation \eqref{Max_ep} and applying integration by parts we obtain
\begin{equation*}
 \int_{\Omega}(\frac{1}{\mu(x)}-1)(\curl E_0(x))\cdot (\curl\overline{E}(x))dx = k^2\int_{\Omega}(\epsilon(x)-1)E_0(x) \cdot \overline{E}(x)dx 
\end{equation*}
\begin{equation}\label{4.3}
 + \int_{\partial\Omega}(\nu \wedge E_0)(x)\cdot(\frac{1}{\mu(x)}\curl\overline{E}(x))ds(x) - \int_{\partial\Omega}(\nu \wedge \overline{E})(x)\cdot(\curl E_0(x))ds(x).
\end{equation}
Therefore combining \eqref{4.2} and  \eqref{4.3} we obtain
\begin{align}
&\int_{\Omega} \frac{1}{\mu(x)}{\vert{\curl\tilde E(x)}\vert}^2dx - \int_{\Omega}(\frac{1}{\mu(x)}-1){\vert \curl E_0(x)\vert}^2dx \nonumber \\
& -k^2\int_{\Omega} \epsilon(x) {\vert \tilde E(x)\vert}^2dx + k^2 \int_{\Omega} (\epsilon(x)-1){\vert E_0(x)\vert}^2dx \nonumber \\
& = \int_{\partial \Omega}(\nu \wedge \overline{E})(x)\cdot(\curl E_0(x))ds(x) - \int_{\partial\Omega} (\nu \wedge E_0)(x)\cdot(\frac{1}{\mu(x)}\curl\overline{E}(x))ds(x) \nonumber \\
& = \int_{\partial \Omega}(\nu \wedge E_0)(x)\cdot(\overline{\curl E_0(x)})ds(x) - \int_{\partial \Omega}(\nu \wedge E_0)(x)\cdot(\overline{\frac{1}{\mu(x)}\curl E(x)})ds(x) \nonumber \\
& = ik \int_{\partial \Omega}(\nu \wedge E_0)(x)\cdot[\overline{(\Lambda_D-\Lambda_{\emptyset})(\nu \wedge E_0)(x)}\wedge \nu(x)]ds(x)  \nonumber \\
& =\tau^{-1} I_{\rho}. \label{4.6*}
\end{align}
\textbf{Step 2}
Now, we show the following identity
\begin{align*}
 &\int_{\Omega} |\curl\tilde E(x)|^2dx - k^2 \int_{\Omega}|\tilde E(x)|^2dx - k^2 \int_{\Omega}(\epsilon(x)-1)| E(x)|^2dx + \int_{\Omega}(\frac{1}{\mu(x)}-1)|\curl E(x)|^2dx\\
 &= -k^2 \int_{\Omega}(\epsilon(x)-1)| E_0(x)|^2dx + \int_{\Omega}(\frac{1}{\mu(x)}-1)|\curl E_0(x)|^2dx + k^2 \int_{\Omega}\epsilon(x)|\tilde E(x)|^2dx \\
 &=- \int_{\Omega}\frac{1}{\mu(x)}|\curl\tilde E(x)|^2dx.
\end{align*}
Replacing $E_0(x)$ by $E(x) - {\tilde E}(x)$ in the equation \eqref{IV}, then we obtain,
\begin{equation}\label{4.5}
 \curl(\frac{1}{\mu(x)}-1)\curl\tilde E(x) +\curl\curl\tilde E(x)-k^2  {\tilde E}(x)   - k^2 (\epsilon(x)- 1)E(x) =0.
\end{equation}
 Multiplying by $\overline{\tilde E}(x)$ in the equation \eqref{4.5} and using integration by parts we obtain,
\begin{equation}\label{4.6}
 \begin{split}
 & \int_{\Omega}[(\frac{1}{\mu(x)}-1)\curl\tilde E(x)]\cdot(\curl\overline{\tilde E}(x))dx + \int_{\Omega} | {\curl\tilde E(x)}|^2dx \\
 & - k^2 \int_{\Omega} {| {\tilde E}(x) |}^2dx - k^2 \int_{\Omega}(\epsilon(x)-1)E(x)\cdot \overline{\tilde E}(x)dx =0.
\end{split}
 \end{equation}
Since, $\nu \wedge {\tilde E}(x) =0$ on the boundary, then we can write the equation \eqref{4.6} as follows
\begin{equation*}
 \int_{\Omega} |\curl\tilde E(x)|^2dx - k^2 \int_{\Omega}| {\tilde E}(x)|^2dx - k^2 \int_{\Omega}(\epsilon(x)-1)| E(x)|^2dx + \int_{\Omega}(\frac{1}{\mu(x)}-1)|\curl E(x)|^2dx
\end{equation*}
\begin{equation}\label{4.7}
 = - \int_{\Omega}[(\frac{1}{\mu(x)}-1)\curl E(x)]\cdot (\curl\overline{\tilde E}(x))dx - k^2 \int_{\Omega}(\epsilon(x)-1)E(x)\cdot \overline{E_0}(x)dx + \int_{\Omega}(\frac{1}{\mu(x)}-1)|\curl E(x)|^2dx.
\end{equation}
Eliminating $E(x)$ by $\tilde E(x) + E_0(x)$ in \eqref{4.7} we get,
\begin{equation*}
 =- k^2 \int_{\Omega}(\epsilon(x)-1) \tilde E(x)\cdot \overline{E_0}(x)dx - k^2 \int_{\Omega}(\epsilon(x)-1)| E_0(x)|^2dx + \int_{\Omega}(\frac{1}{\mu(x)}-1)(\curl\tilde E)(x)\cdot(\curl\overline{E_0}(x))dx 
\end{equation*}
\begin{equation}\label{4.8}
 + \int_{\Omega}(\frac{1}{\mu(x)}-1)|\curl E_0(x)|^2dx.
\end{equation}
Again from the equation \eqref{IV} taking complex conjugate, we can write,
\begin{equation}\label{4.9}
 \curl\frac{1}{\mu(x)}\curl\overline{\tilde E}(x) +\curl(\frac{1}{\mu(x)}-1)\curl\overline{E_0}(x) -k^2 \epsilon(x) \overline{\tilde E}(x) - k^2 (\epsilon(x)- 1){\overline{E} }_0(x) = 0.
\end{equation}
Multiplying by $\tilde E(x)$ in the equation \eqref{4.9} and using integration by parts formula the following equality follows
\begin{equation*}
 \int_{\Omega} \frac{1}{\mu(x)} |\curl\tilde E(x) |^2dx + \int_{\Omega}(\frac{1}{\mu(x)}-1)(\curl\overline{E_0}(x))\cdot(\curl\tilde E(x))dx- k^2 \int_{\Omega} \epsilon(x)|\tilde E(x)|^2dx 
\end{equation*}
\begin{equation}\label{4.10}
 -k^2 \int_{\Omega}(\epsilon(x)-1)\overline{E_0}(x)\cdot \tilde E(x) dx= 0.
\end{equation}
Hence, from the equations \eqref{4.8} and \eqref{4.10} we can get ,
 \begin{align}
& \int_{\Omega} |\curl\tilde E(x)|^2dx - k^2 \int_{\Omega}|\tilde E(x)|^2dx - k^2 \int_{\Omega}(\epsilon(x)-1)| E(x)|^2dx + \int_{\Omega}(\frac{1}{\mu(x)}-1)|\curl E(x)|^2dx \nonumber \\
& = -k^2 \int_{\Omega}(\epsilon(x)-1)| E_0(x)|^2dx + \int_{\Omega}(\frac{1}{\mu(x)}-1)|\curl E_0(x)|^2dx \nonumber \\
& + k^2 \int_{\Omega}\epsilon(x) |\tilde E(x)|^2dx -
 \int_{\Omega}\frac{1}{\mu(x)}|\curl\tilde E(x)|^2dx \nonumber \\
& = -\tau^{-1} I_\rho(\tau,t). \label{4.11}
 \end{align}
Combining \eqref{4.11} with the formula
\begin{equation*}
 |\curl\tilde E(x) |^2 + (\frac{1}{\mu(x)}-1)|\curl E(x) |^2 = \frac{1}{\mu(x)}|\curl E(x) - \mu(x)(\curl E_0)(x)|^2 + (1-\mu(x))|\curl E_0(x)|^2
\end{equation*}
we obtain
\begin{equation}\label{4.12}
 -\tau^{-1} I_{\rho}(\tau,t) \geq \int_{\Omega}(1-\mu(x))|\curl E_0(x)|^2dx - k^2 \int_{\Omega} |\tilde E(x)|^2dx -k^2 \int_{\Omega}(\epsilon(x)-1)| E(x)|^2dx.
\end{equation}
Finally, again from \eqref{4.11}, we have
\begin{equation}
 \tau^{-1}I_{\rho}(\tau,t) \geq \int_{D}(1-\frac{1}{\mu(x)})\vert \curl E_0(x)\vert^2dx - k^2\int_{\Omega}\epsilon(x)\vert \tilde E(x)\vert^2dx + k^2\int_D(\epsilon(x)-1)\vert E_0(x)\vert^2dx.
\end{equation}
\end{proof}
\subsection{Estimates of the lower order term $\tilde{E}$}
\begin{proposition}\label{pro1}
Assume that $\Omega$ is $C^1$-smooth and $D$ is a subset strictly included in $\Omega.$ Then there exist a positive constant $C $ independent of $\tilde E$ 
and $E_0$ and a positive constant $\delta$ depending only on $\Omega$ such that we have:
\begin{equation*}
 {\Vert \tilde E\Vert}_{L^2(\Omega)} \leq C \{ {\Vert \curl E_0\Vert}_{L^p(D)} + {\Vert E_0\Vert}_{L^2(D)}\}
\end{equation*}
for every $p$ in $(\max\{ \frac{4}{3},\frac{2+\delta}{1+\delta}\},2]$.
\end{proposition}
\begin{proof}
 Set $f := -(\frac{1}{\mu(x)}-1)\curl E_0$, $g := k^2 (\epsilon(x)- 1)E_0.$ Then the reflected solution $\tilde E$ satisfies
\begin{equation}\label{ne}
\begin{cases}
  &\curl (\frac{1}{\mu(x)}\curl \tilde E)-k^2 \epsilon(x)\tilde E = \curl f + g
\ \ \text{in}\ \Omega, \\
&\nu\wedge \tilde E=0 \ \ \text{on}\ \partial \Omega.
\end{cases}
\end{equation}
From the $L^p$-estimate, see Theorem \ref{5.6} in Section 4, the following problem 
\begin{equation}
\begin{cases}
  &\curl (\frac{1}{\mu(x)}\curl U)+ (\sup_{x\in\Omega}{\frac{1}{\mu(x)}}) U = \curl f
\ \ \text{in}\ \Omega, \\
&\nu\wedge U =0 \ \ \text{on}\ \partial \Omega,
\end{cases}
\end{equation}
has a unique solution in $H_{0}^{1,p}(\curl;\Omega)$ with the estimate
\begin{equation}\label{44.13}
 {\Vert U\Vert}_{L^p(\Omega)} + {\Vert \curl U\Vert}_{L^p(\Omega)} \leq C{\Vert f\Vert}_{L^p(\Omega)}
\end{equation}
for $p\in (\frac{2+\delta}{1+\delta},2]$ for some $\delta > 0$, depending only on $\Omega$.\\
We set $\mathcal{E} = \tilde E - U,$ then $\mathcal{E}$ satisfies
\begin{equation}\label{E}
\begin{cases}
&\curl (\frac{1}{\mu(x)}\curl \mathcal{E})- k^2 \epsilon(x)\mathcal{E} = (k^2 \epsilon(x) + \sup_{x\in\Omega}{\frac{1}{\mu(x)}}) U  + g
\ \ \text{in}\ \Omega, \\
&\nu\wedge \mathcal{E} =0 \ \ \text{on}\ \partial \Omega.
\end{cases}
\end{equation}
By the well-posedness of \eqref{E} in $H(\curl;\Omega)$, see \cite{Monk}, we obtain
\[
\|\mathcal{E}\|_{L^2(\Omega)} + \|\curl \mathcal{E}\|_{L^2(\Omega)}  \leq C\{\|U\|_{L^2(\Omega)} + \|g\|_{L^2(\Omega)}\}
\]
and in particular, we have for $p \leq 2$
\begin{equation}\label{44.14}
\|\mathcal{E}\|_{L^p(\Omega)} + \|\curl \mathcal{E}\|_{L^p(\Omega)}  \leq C\{\|U\|_{L^2(\Omega)} + \|g\|_{L^2(\Omega)}\}.
\end{equation}
We recall again the Sobolev-Besov space $B_{\frac{1}{p}}^{p,2}(\Omega).$
Then Property \ref{pr7} in Appendix implies that $U\in B_{\frac{1}{p}}^{p,2}(\Omega)$, for $1<p\leq 2$, since both $U$ and $\curl \; U$ are in $L^p(\partial \Omega)$ and 
$\nabla\cdot U = 0$ and $\nu\wedge U = 0$ on $\partial \Omega$.
As the inclusion map $B_{\frac{1}{p}}^{p,2}(\Omega) \rightarrow L^2(\Omega)$ is continuous for $p \in (\frac{4}{3} ,2]$, see [\cite{T.Muramatu}, Theorem 2],
we have the estimate
\begin{equation}\label{44.15}
 \|U\|_{L^2(\Omega)} \leq C \|U\|_{B_{\frac{1}{p}}^{p,2}(\Omega)}.
\end{equation}
Again since $\nabla\cdot U = 0$ in $\Omega$ and $\nu\wedge U = 0$ on $\partial \Omega$, then from Property \ref{pr7} in Appendix 
together with \eqref{44.15} we obtain
\begin{equation}\label{44.16}
 \|U\|_{L^2(\Omega)} \leq C \{ \|U\|_{L^p(\Omega)} + \|\curl U\|_{L^p(\Omega)}\}
\end{equation}
for $p \in (\frac{4}{3} ,2].$\\
Combining \eqref{44.13} with \eqref{44.14} and \eqref{44.16}, we obtain
\begin{equation}\label{44.17}
 \|\mathcal{E}\|_{L^p(\Omega)} + \|\curl \mathcal{E}\|_{L^p(\Omega)}  \leq  C\{ {\| f\|}_{L^p(\Omega)} + {\| g\|}_{L^2(\Omega)} \}
\end{equation}
for $p \in (\max\{\frac{4}{3}, \frac{2+\delta}{1+\delta} \},2].$\\
     Since $\tilde{E} = \mathcal{E} + U$, then from \eqref{44.13} and \eqref{44.17} we have
\begin{equation}\label{44.18}
\|\tilde E\|_{L^p(\Omega)} + \|\curl \tilde E \|_{L^p(\Omega)}  \leq  C\{ {\| f\|}_{L^p(\Omega)} + {\| g\|}_{L^2(\Omega)} \}.
\end{equation}
  Therefore, recalling again Property \ref{pr7} in Appendix A and combining it with the inequality \eqref{44.15} we have
\begin{equation}\label{lain}
\begin{split}
\|\tilde E\|_{L^2(\Omega)} 
&\leq C \|\tilde E\|_{B_{\frac{1}{p}}^{p,2}(\Omega)} \\
&  \leq C \{\|\tilde E\|_{L^p(\Omega)} + \|\curl \tilde E\|_{L^p(\Omega)} + \|\nabla\cdot\tilde{E}\|_{L^p(\Omega)}\},  \ \ \text{since}\ \tilde E\wedge\nu=0. 
\end{split}
\end{equation}
On the other hand from \eqref{ne}, we have $k^2\nabla\cdot(\epsilon\tilde E) = \nabla\cdot g$. 
Recall that $\nabla\cdot g = k^2\{ (\nabla(\epsilon-1))\cdot E_0 + (\epsilon-1)(\nabla\cdot E_0)\}$. However, $\nabla\cdot E_0 = 0$ then, we have
$\nabla\epsilon\cdot\tilde E + \epsilon(\nabla\cdot\tilde E) = \nabla\epsilon\cdot E_0$.
 Therefore as $\epsilon \in W^{1, \infty}(\Omega)$, we deduce the estimate
\begin{equation}\label{epsi}
 \|\nabla\cdot\tilde{E}\|_{L^p(\Omega)} \leq \frac{\|\nabla\epsilon\|_{L^{\infty}(D)}}{\|\epsilon\|_{L^{\infty}(D)}} \{\|\tilde E\|_{L^p(\Omega)} + \|E_0\|_{L^p(D)} \}.
\end{equation}
 Combining the inequalities \eqref{44.18}, \eqref{lain} and \eqref{epsi}, we obtain 
\[
\begin{split}
\|\tilde E\|_{L^2(\Omega)}
& \leq C \{\|\tilde E\|_{L^p(\Omega)} + \|\curl \tilde E\|_{L^p(\Omega)} + \|E_0\|_{L^p(D)}\} \\
& \leq C \{ {\| f\|}_{L^p(\Omega)} + {\| g\|}_{L^2(\Omega)} + \|E_0\|_{L^p(D)} \} \\
&\leq C \{ \|\curl E_0\|_{L^p(D)} + \|E_0\|_{L^2(D)}\}
\end{split}
\]
for $p\in (\max\{\frac{4}{3}, \frac{2+\delta}{1+\delta} \},2].$      
\end{proof}
Note that electric part of the CGO solution defined in \eqref{CGO-formImpeneq} is nothing but a multiplication by a constant of the magnetic
part of the CGO solution defined in \eqref{CGO-formPerTT}. 
Then, similarly the magnetic part of the CGO solution defined in \eqref{CGO-formImpeneq} is nothing but a multiplication by a constant of the electric
part of the CGO solution defined in \eqref{CGO-formPerTT}.
So, we have the following lemmas for the CGOs in \eqref{CGO-formPerTT} in the same way as we did in Lemma \ref{lemma3.9}, Lemma \ref{lemma3.10} and Lemma \ref{lemma3.11}.
\begin{lemma}\label{lemma3.9Pene}
For $ 1 \leq q < \infty $, The following estimates hold. 
\bigskip

\noindent 1.  
\begin{equation}\label{3.23P}
 \int_D {| {E_0(x)} |}^qdx \leq C{\tau}^{-1} \sum_{j=1}^{n}\iint_{{| y'|} < \delta} e^{-q\tau l_j(y')}dy' - \frac{C}{q}\tau^{-1} e^{-q \delta \tau}
+ C e^{-qc \tau}
\end{equation}
\bigskip
\noindent 2.
\begin{equation}\label{3.24P}
 \int_D {| {E_0(x)} |}^2dx \geq C{\tau}^{-1} \sum_{j=1}^{n}\iint_{{| y'|} < \delta} e^{-2\tau l_j(y')}dy' - \frac{C}{2}\tau^{-1} e^{-2 \delta \tau}
\end{equation}
\bigskip
\noindent 3.
\begin{equation}\label{3.25P}
 \int_D {| {H_0(x)} |}^qdx \leq C{\tau}^{q-1} \sum_{j=1}^{n}\iint_{{| y'|} < \delta} e^{-q\tau l_j(y')}dy' - \frac{C}{q}\tau^{q-1} e^{-q \delta \tau}
+ C \tau^q e^{-qc \tau}
\end{equation}
\bigskip
\noindent 4.
\begin{equation}\label{3.26P}
 \int_D {| {H_0(x)} |}^2dx \geq C{\tau} \sum_{j=1}^{n}\iint_{{| y'|} < \delta} e^{-2\tau l_j(y')}dy' - \frac{C}{2}\tau e^{-2 \delta \tau}.
\end{equation}
\end{lemma}
\begin{lemma}\label{lemma3.10Pene}
We have the following estimate
\begin{equation*}
 \frac{\|H_0\|_{L^2(D)}^{2}}{\|E_0\|_{L^2(D)}^{2}} \geq \mathcal{O}(\tau^{2}), \ \ \tau\gg1. 
\end{equation*}
\end{lemma}
\begin{lemma}\label{lemma3.11Pene}
If $t = h_D(\rho),$ then for some positive constant C,
\begin{equation*}
 \liminf_{\tau\rightarrow \infty} \int_{D}\tau |\curl E_0(x)|^2 dx \geq C.
\end{equation*}
\end{lemma}
 \begin{lemma}\label{lemma4.3}
  For  $p\in (\max\{\frac{4}{3}, \frac{2+\delta}{1+\delta} \},2]$, we have the following estimates
\begin{equation}\label{compare}
 \frac{\|\tilde E\|_{L^2(\Omega)}^2}{\|\curl E_0\|_{L^2(D)}^2} \leq C\tau^{1-\frac{2}{p}} \ \ (\tau \gg 1).
\end{equation}
 \end{lemma} 
\begin{proof}
From Proposition \ref{pro1}, we have 
\begin{equation*}
\|\tilde E\|_{L^2(\Omega)} \leq C\{ \|\curl E_0\|_{L^p(D)} + \|E_0\|_{L^2(D)} \}.
\end{equation*}
Similarly as in the proof of Lemma \ref{neiii}, we obtain 
\begin{equation}\label{SIIIIMIla}
 \frac{\|\curl E_0\|_{L^p(D)}^{2}}{\|\curl E_0\|_{L^2(D)}^{2}} \leq C \tau^{1-\frac{2}{p}}, \ \ \tau\gg1.
\end{equation}
for all $p\leq2$.
Therefore combining Lemma \ref{lemma3.10Pene} and \eqref{SIIIIMIla}, we obtain
 \begin{align*}
 \frac{\|\tilde E\|_{L^2(\Omega)}^2}{\|\curl E_0\|_{L^2(D)}^2} 
& \leq C\left[\frac{\|\curl E_0\|_{L^p(D)}^2}{\|\curl E_0\|_{L^2(D)}^2} 
+ \frac{\|E_0\|_{L^2(D)}^2}{\|\curl E_0\|_{L^2(D)}^2}\right] \\
&  \leq C\{ \tau^{1-\frac{2}{p}} + \tau^{-2} \} \\
&  \leq C\tau^{1-\frac{2}{p}} \ \ (\tau \gg 1).
 \end{align*}
\end{proof}
\textbf{End of the proof of Theorem \ref{theorem1}}\\
\textbf{Case 1.} $1-\mu(x)>C>0$.\\
 From the first inequality in Lemma \ref{lemma4.1} we have
\begin{align*}
 -I_{\rho}(\tau,t) 
&\geq \tau\int_{D}(1-\mu(x))\vert \curl E_0(x)\vert^2 dx- \tau C \int_{\Omega} \vert \tilde E(x)\vert^2 dx - \tau C\int_{D} \vert E_0(x)\vert^2 dx \\
& \geq C \tau\int_{D}\vert \curl E_0(x)\vert^2 dx- \tau C \int_{\Omega} \vert \tilde E(x)\vert^2 dx - \tau C\int_{D} \vert E_0(x)\vert^2 dx. 
\end{align*}
Using the inequality in Lemma \ref{lemma4.3} and choosing $p$ in $(\max\{\frac{4}{3}, \frac{2+\delta}{1+\delta}\},2)$ we obtain
\begin{align*} 
\frac{-I_{\rho}(\tau,t)}{{\Vert \curl E_0\Vert}_{L^2(D)}^{2}} 
&\geq C{\tau}\left[ 1 - \frac{\int_{\Omega}{\vert \tilde E(x)\vert}^2dx}{\int_{D}{\vert \curl E_0(x) \vert}^2dx} - \frac{\int_{D} \vert E_0(x)\vert^2 dx}{\int_{D}{\vert \curl E_0(x) \vert}^2dx}\right] \\
& \geq C \tau \{ 1 - \tau^{1-\frac{2}{p}} - 2\tau^{-2}\}. 
\end{align*}
Hence, using Lemma \ref{lemma3.11Pene} we deduce that for $\tau >> 1,$
\begin{equation*}
  \vert I_{\rho}(\tau,h_D(\rho))\vert \geq C >0
\end{equation*}
which ends the proof.  \\   
\textbf{Case 2.} $1-\mu(x)<-C<0$.\\
Similarly, from the second inequality in Lemma \ref{lemma4.1} we obtain
\begin{equation*}
 \tau^{-1}I_{\rho}(\tau,t) \geq \int_{D}(1-\frac{1}{\mu(x)})\vert \curl E_0(x)\vert^2dx - k^2\int_{\Omega}\vert \tilde E(x)\vert^2dx + C \int_{D} \vert E_0(x)\vert^2 dx.
\end{equation*}
Then using the same argument as in \textbf{Case 1} we end the proof.
 \section{An $L^p$-type estimate for the solutions of the Maxwell system}
 Let us consider a general time harmonic Maxwell systems of equations of the form
\begin{equation}\label{The-general-model}
 \curl (A(x)\curl E(x)) + M E(x) = \curl f(x) + g(x)
\end{equation}
where $A$ is a Hermitian matrix, with measurable entries, satisfying the uniformly ellipticity condition, i.e, there exists positive constants $\lambda,$
$ M $ such that
\begin{equation}\label{uni}
 \lambda {\vert \xi\vert}^2 \leq A(x)\xi \cdot \overline{\xi}  \leq M {\vert \xi\vert}^2
\end{equation}
for all $\xi \in {\mathbb C}^3$ and for almost every $x$ in $\Omega$. 
 Let $\Omega $ be a bounded and $C^1$ smooth domain. We recall that the space 
$
 H_0(\curl;\Omega) =\{ E\in L^2(\Omega) : \curl E \in L^2(\Omega), \nu \wedge E = 0 \hspace{.15cm} \mbox{on} \hspace{.15cm}\partial \Omega\}
$
is a Banach space under the norm 
$
 {\Vert E\Vert}_{H_0(\curl;\Omega)} := {\Vert E\Vert}_{L^2(\Omega)} + {\Vert \curl E\Vert}_{L^2(\Omega)}.
$
 Now, we define the space ${H_{0}^{1,q}}(\curl;\Omega)$ as
$
 {H_{0}^{1,q}}(\curl;\Omega) :=\{ E\in L^q(\Omega) : \curl E \in L^q(\Omega), \nu \wedge E = 0 \hspace{.15cm} \mbox{on} \hspace{.15cm}\partial \Omega\}
$
for $1 < q < \infty.$
The norm of this space is defined as 
$
 {\Vert E\Vert}_{{H_{0}^{1,q}}(\curl;\Omega)} := {\Vert E\Vert}_{L^q(\Omega)} + {\Vert \curl E\Vert}_{L^q(\Omega)}
$
and an equivalent norm is given by
$
 {\Vert E\Vert}_{1,q} := ({\Vert E\Vert}_{L^q(\Omega)}^{q} + {\Vert \curl E\Vert}_{L^q(\Omega)}^{q})^{\frac{1}{q}}.
$
Under this second norm, ${H_{0}^{1,q}}(\curl;\Omega)$ is also a Banach space.
\bigskip

The object of this section is to prove that the solution operator corresponding to the problem given by (\ref{The-general-model})
in ${H_{0}^{1,p}}(\curl;\Omega)$ is invertible for $p$ near $2$. An idea to prove this property is to use a perturbation argument.
Precisely, first we show that in the case $A=I$, $I$ is the identity matrix, and $M=1$, this property is true for $p$ in an interval containing
 $2$ (in our case this interval is $(1, \infty)$). Then using an equivalent variational formulation
 of (\ref{The-general-model}) in ${H_{0}^{1,p}}(\curl;\Omega)$ and a perturbation argument, we aim at proving the same property for
a general matrix $A$ but for $p$ near $2$. In the scalar divergence form elliptic problems, these arguments have been successfully 
applied. There are several methods to justify this perturbation argument. We cite the original one by Meyers, 
see \cite{Me}, which works for linear and complex perturbations , i.e. $A$ is a linear matrix with possibly complex entries. 
We cite also the method by Gr{\"o}ger, see \cite{Groe}, which works for real valued and nonlinear perturbations of the leading term but justified only for 
$p\geq 2$. Recently the argument of Gr{\"o}ger has been generalized to the elasticity system, see \cite{HMW}. Here, we generalize the method by Meyers to the Maxwell system to deal with linear perturbations but allowing $p$ to be near $2$ and $p\leq 2$
, see Theorem \ref{5.6}. It is this last property that is needed in our analysis in Section 3. 
In \cite{HMW}, other perturbative methods are also discussed.
\bigskip

To start, we  define the bilinear form
\begin{equation*}
 \mathcal{B}_{A}(E,F) := \int_{\Omega} A(x) \curl E(x) \cdot \overline{\curl F(x)}dx +  
M\int_{\Omega} E(x) \cdot \overline{F(x)}dx,
\end{equation*}
for all $
 E\in {H_{0}^{1,q}}(\curl;\Omega)$ and $F\in{H_{0}^{1,{q'}}}(\curl;\Omega),
$
 where $\frac{1}{q} + \frac{1}{q'} = 1.$
 For $A(x)= I, $ we set $\mathcal{B} := \mathcal{B}_{A}$.
 Consider the problem
\begin{equation}\label{57}
 \curl (A(x)\curl E(x)) + M E(x) = \curl f(x) + g(x) \hspace{.25cm} \text{on} \hspace{.25cm} \Omega
\end{equation}
where  $f $ and $ g$ are in $L^q(\Omega),$ $q\in (1, +\infty).$ 
Recall that $E\in {H_{0}^{1,q}}(\curl;\Omega)$ is a weak solution of the equation \eqref{57} if we have
\begin{equation}\label{58}
 \mathcal{B}_{A}(E,F) = \int_{\Omega} f(x)\cdot \overline{\curl F(x)} dx + \int_{\Omega} g(x)\cdot \overline{F(x)}dx
\end{equation}
for all $F \in {H_{0}^{1,q'}}(\curl;\Omega)$.\\
Dividing by $M$ in both sides of \eqref{uni}, we reduce our study to the case $M = 1,$ i.e. we have
\begin{equation}\label{unii}
 \lambda {\vert \xi\vert}^2 \leq A(x)\xi \cdot \overline{\xi}  \leq {\vert \xi\vert}^2.
\end{equation}
\begin{subsection}{\textbf{The imperturbed problem}}
\end{subsection}
\begin{theorem}\label{Lemma5.2}
Let $\Omega$ be a bounded $C^1$ domain in $\mathbb{R}^3$. Then for $f \in L^p(\Omega)$ and $g \in L^p(\Omega)$ with $1 < p< \infty$, 
the boundary value problem 
\begin{equation}\label{V}
\begin{cases}
&\curl \curl E + E = \curl f + g
\ \ \text{in}\ \Omega,\\
&\nu \wedge E = 0 \ \ \text{on}\ \partial \Omega,
\end{cases}
\end{equation}
is uniquely weakly solvable and there exists $C = C(p,k,\Omega) > 0$ such that,
\begin{equation*}
 {\Vert E\Vert}_{L^p(\Omega)} + {\Vert \curl E\Vert}_{L^p(\Omega)} \leq C \{{\Vert f\Vert}_{L^p(\Omega)} +{\Vert g\Vert}_{L^p(\Omega)}\}.
\end{equation*}
\end{theorem}
\begin{proof}
 Since the problem related to this Maxwell system is self-adjoint, then $k= i$ is not a Maxwell eigenvalue. 
Therefore the rest of the proof follows from Property \ref{pr2} in Appendix. 
\end{proof}
Based on Theorem \ref{Lemma5.2}, we can define the linear transformation $T_I$ as follows
\begin{equation*}
T_I : L^p(\Omega) \times L^p(\Omega) \rightarrow  L^p(\Omega) \times L^p(\Omega) \hspace{.15cm} by
\end{equation*}
\begin{equation}
(f,g)\mapsto (\curl E, E) \ \ \text{(solution of the problem {\eqref{V}})}. 
\end{equation}
We denote its norm by ${\Vert T_I\Vert}_p.$ The next two lemmas give some properties of this norm in terms of $p \in (1, \infty)$.
Note that sometimes we use the notation $L^p := L^p(\Omega)$ without writing $\Omega$ to avoid heavy notations.
\begin{lemma}\label{5.4}
 We have ${\Vert T_I \Vert}_2 = 1.$ 
\end{lemma}
\begin{proof}On the one hand, we have
\begin{align}
 {\Vert T_I \Vert}_2 
& = \sup_{{\Vert (f,g) \Vert}_{L^2 \times L^2} = 1} {\Vert T_I(f,g) \Vert}_{L^2 \times L^2} \nonumber \\
& = \sup_{{\Vert (f,g) \Vert}_{L^2 \times L^2} = 1} {\Vert (\curl U^{f,g}, U^{f,g}) \Vert}_{L^2 \times L^2} \nonumber \\
& = \sup_{{\Vert (f,g) \Vert}_{L^2 \times L^2} = 1} \sup_{{\Vert (\tilde f,\tilde g \Vert}_{L^2 \times L^2} = 1}{\vert \int_\Omega \curl U^{f,g}(x) \cdot \overline{\tilde f(x)}dx + 
\int_\Omega U^{f,g}(x) \cdot \overline{\tilde g(x)}dx \vert} \nonumber \\
& \geq \sup_{{\Vert (f,g) \Vert}_{L^2 \times L^2} = 1} \sup_{{\Vert (\curl \tilde g,\tilde g) \Vert}_{L^2 \times L^2} = 1}
{\vert \int_\Omega \curl U^{f,g}(x) \cdot \overline{\curl \tilde g(x)}dx 
+ \int_\Omega U^{f,g}(x) \cdot \overline{\tilde g(x)}dx \vert} \nonumber \\
& \big( \mbox{ using } (\ref{58}) \mbox{ with } A=I \big)\; \nonumber \\
& = \sup_{{\Vert (f,g) \Vert}_{L^2 \times L^2} = 1} \sup_{{\Vert (\curl \tilde g,\tilde g) \Vert}_{L^2 \times L^2} = 1}
{\vert \int_\Omega (f(x), g(x))\cdot \overline{(\curl \tilde g(x), \tilde g(x))}dx \vert}\;  \nonumber \\
& \big( 
\forall \tilde g \in {H_{0}^{1,2}}(\curl;\Omega)\; /\; {\Vert (\curl \tilde g,\tilde g) \Vert}_{L^2 \times L^2} = 1 \big) \nonumber \\
& \geq \sup_{{\Vert (f,g) \Vert}_{L^2 \times L^2} = 1} 
{\vert \int_\Omega (f(x), g(x))\cdot \overline{(\curl \tilde g(x), \tilde g(x))}dx \vert}\;  \nonumber \\
& = {\Vert (\curl \tilde g,\tilde g) \Vert}_{L^2 \times L^2}\;  \big( 
\forall \tilde g \in {H_{0}^{1,2}}(\curl;\Omega)\; /\; {\Vert (\curl \tilde g,\tilde g) \Vert}_{L^2 \times L^2} = 1 \big) \nonumber \\
&=1. \label{one_side}
\end{align}
On the other hand,
\begin{equation*}
 {\Vert T_I \Vert}_2 = \sup_{{\Vert (f,g) \Vert}_{L^2 \times L^2} = 1} \sup_{{\Vert (\tilde f,\tilde g) \Vert}_{L^2 \times L^2} = 1}
{\vert \int_\Omega \curl U^{f,g}(x) \cdot \overline{\tilde f(x)}dx + \int_\Omega U^{f,g}(x) \cdot \overline{\tilde g(x)}dx \vert}
\end{equation*}
and the Cauchy-Schwartz inequality gives
\begin{equation}\label{Ch}
\begin{split}
 {\Vert T_I \Vert}_2  
&\leq \sup_{{\Vert (f,g) \Vert}_{L^2 \times L^2} = 1} \sup_{{\Vert (\tilde f,\tilde g) \Vert}_{L^2 \times L^2} = 1} [{\Vert \curl U^{f,g}\Vert}_{L^2}{\Vert\tilde f\Vert}_{L^2} + {\Vert U^{f,g}\Vert}_{L^2} {\Vert \tilde g\Vert}_{L^2}] \\
& \leq \sup_{{\Vert (f,g) \Vert}_{L^2 \times L^2} = 1} [\frac{1}{2} \{ {\Vert \curl U^{f,g}\Vert}_{L^2}^{2} + {\Vert U^{f,g}\Vert}_{L^2}^{2}\} + \frac{1}{2}]. \\
\end{split}
\end{equation}
Given $f\in L^2(\Omega)$ and $g\in L^2(\Omega)$, $ U^{f,g} \in {H_{0}^{1,2}}(\curl;\Omega)$ is a weak solution of the equation
\begin{equation*}
 \curl \curl U^{f,g} + U^{f,g} = \curl f + g,
\end{equation*}
i.e. it satisfies (\ref{58}) for $A=I$ and $q=2$. Taking $E=F=U^{f,g}$ in (\ref{58}), we obtain
\begin{equation*}
 \int_{\Omega} {\vert \curl U^{f,g}(x)\vert}^2dx + \int_{\Omega} {\vert U^{f,g}(x)\vert}^2dx = 
\int_{\Omega} f(x) \cdot \overline{\curl U^{f,g}(x)}dx + \int_{\Omega} g(x) \cdot \overline{U^{f,g}(x)}dx.
\end{equation*}
By the Cauchy-Schwartz and Young inequalities, we deduce that
\begin{equation}\label{5.13}
 {\Vert \curl U^{f,g}\Vert}_{L^2}^{2} + {\Vert U^{f,g}\Vert}_{L^2}^{2} \leq {\Vert f\Vert}_{L^2}^{2} + {\Vert g\Vert}_{L^2}^{2}.
\end{equation}
So, \eqref{Ch} and \eqref{5.13} give 
\begin{equation}\label{5.14}
 {\Vert T_I \Vert}_2  \leq 1.
\end{equation}
 Finally \eqref{one_side} and \eqref{5.14} imply 
\begin{equation}
 {\Vert T_I \Vert}_2  = 1.
\end{equation}  
\end{proof}
The following lemma gives a lower bound of the bilinear form $\mathcal{B}$ in terms of $\Vert T_I \Vert_p$.
\begin{lemma}\label{5.5}
 For $2 \leq p < \infty,$ we have
\begin{equation*}
 \inf_{{\Vert F\Vert}_{1,p'} = 1} \sup_{{\Vert E\Vert}_{1,p} = 1} \vert \mathcal{B}(E,F)\vert \geq \frac{1}{{\Vert T_I \Vert}_p}.
\end{equation*}
\end{lemma}  
\begin{proof}
Let $F\in H_{0}^{1,{p'}}(\curl; \Omega)$, then
 \begin{align*}
 {\Vert (\curl F, F)\Vert}_{L^{p'} \times L^{p'}} 
& = \sup_{{\Vert (\tilde f,\tilde g) \Vert}_{L^p \times L^p} = 1} {\vert (\curl F, F) \cdot (\tilde f, \tilde g)\vert} \\
& = \sup_{{\Vert (\tilde f,\tilde g) \Vert}_{L^p \times L^p} = 1} {\vert \int_{\Omega} \{ \curl F(x) \cdot \overline{\tilde f(x)} + F(x) \cdot \overline{\tilde g(x)}\}dx\vert} 
\end{align*}
where we used the fact that $(L^p(\Omega)\times L^p(\Omega))^{'}$ is isometrically isomorphic to $L^{p'}(\Omega)\times L^{p'}(\Omega) $, 
see [\cite{Simader}, Lemma 4.2], for instance.
In addition, for $\tilde f \in L^p(\Omega)$ and $\tilde g \in L^p(\Omega)$ there exists a unique $\tilde E \in H_{0}^{1,{p}}(\curl; \Omega) $
such that (\ref{58}), with  $A=I$, is satisfied. Therefore, $\mbox{ using } (\ref{58}) \mbox{ with } A=I$
\[
{\Vert (\curl F, F)\Vert}_{L^{p'} \times L^{p'}}
 = \sup_{{\Vert (\tilde f,\tilde g) \Vert}_{L^p \times L^p} = 1} 
{\vert \int_{\Omega} \{ \curl F(x) \cdot \overline{\curl \tilde E(x)} + F(x) \cdot \overline{\tilde E(x)}\}dx\vert}. 
\]
The boundedness property of $T_I : L^p(\Omega) \times L^p(\Omega) \rightarrow L^p(\Omega) \times L^p(\Omega)$ for $p \geq 2$ implies
\begin{align*}
 {\Vert (\curl F, F)\Vert}_{L^{p'} \times L^{p'}}
& \leq \sup_{{\Vert (\curl \tilde E,\tilde E) \Vert}_{L^p \times L^p} \leq {\Vert T_I\Vert}_p} {\vert \int_{\Omega} 
\{\curl F(x)\cdot \overline{\curl \tilde E(x)} + F(x) \cdot \overline{\tilde E(x)}\}dx\vert} \\
& = {\Vert T_I\Vert}_p \sup_{{\frac{{\Vert (\curl \tilde E,\tilde E) \Vert}_{L^p \times L^p}}{{\Vert T_I\Vert}_p}} \leq 1}
 {\vert \int_{\Omega} \{\curl F(x)\cdot \overline{\frac{\curl \tilde E(x)}{{\Vert T_I\Vert}_p}} + F(x) \cdot \overline{\frac{\tilde E(x)}{{\Vert T_I\Vert}_p}}\}dx\vert} 
\end{align*}
Define $\tilde{\tilde E} := \frac{\tilde E}{{\Vert T_I\Vert}_p}.$
Then we get,
\begin{equation*}
 {\Vert (\curl F, F)\Vert}_{L^{p'} \times L^{p'}} \leq {\Vert T_I\Vert}_p\sup_{{\Vert (\curl \tilde {\tilde E},\tilde {\tilde E}) \Vert}_{L^p \times L^p} \leq 1} 
{\vert \int_{\Omega} \{\curl F(x) \cdot \overline{\curl \tilde{\tilde E}(x)} + F(x) \cdot \overline{\tilde{\tilde E}(x)\}}dx\vert}
\end{equation*}
Now, define $E := \frac{\tilde{\tilde E}}{{\Vert (\curl \tilde {\tilde E},\tilde {\tilde E}) \Vert}_{L^p \times L^p}}$.
Therefore we obtain
\begin{equation*}
 {\Vert (\curl F, F)\Vert}_{L^{p'} \times L^{p'}} \leq {\Vert T_I\Vert}_p\sup_{{\Vert (\curl E, E) \Vert}_{L^p \times L^p} = 1} 
{\vert \int_{\Omega} \{\curl F(x) \cdot \overline{\curl E(x)} + F(x) \cdot \overline{E(x)}\}dx\vert}
\end{equation*}
Since ${\Vert (\curl E, E) \Vert}_{L^p \times L^p} = {\Vert E\Vert}_{1,p}$, then taking infimum over $H_{0}^{1,{p'}}(\curl; \Omega)$ we obtain
\begin{equation}\label{5.16*}
 \inf_{{\Vert F\Vert}_{1,p'} = 1} \sup_{{\Vert E\Vert}_{1,p} = 1} \vert \mathcal{B}(E,F)\vert \geq \frac{1}{{\Vert T_I \Vert}_p}.
\end{equation} 
\end{proof}
\begin{subsection}{\textbf{The perturbed problem}}
\end{subsection}
\begin{lemma}\label{5.3}
 Let $\Omega$ be a bounded $C^1$ domain. Suppose that $A= A(x)$ is a Hermitian matrix, with measurable entries, and satisfies the uniformly 
ellipticity condition \eqref{unii}. 
Assume that $q$ is some fixed number satisfying $2 \leq q< \infty.$
Under the condition
\begin{equation}\label{5.16}
 \inf_{{\Vert F\Vert}_{1,q'}=1} \sup_{{\Vert E\Vert}_{1,q}=1} {\vert {\mathcal{B}}_{A}(E,F)\vert}\geq \frac{1}{K} >0
\end{equation}
 the Maxwell system of equations 
\begin{equation}\label{5.17}
 \curl (A\ \curl E) + E = \curl f + g
\end{equation}
is uniquely weakly solvable in $ H_{0}^{1,q'}(\curl;\Omega)$ for each $g\in{L^{q'}(\Omega)}$ and
 $f\in{L^{q'}(\Omega)}$ and the weak solution satisfies
\begin{equation*}
 {\Vert E\Vert}_{L^{q'}(\Omega)} + {\Vert \curl E\Vert}_{L^{q'}(\Omega)} \leq K \{ {\Vert f\Vert}_{L^{q'}(\Omega)} +{\Vert g\Vert}_{L^{q'}(\Omega)}\},
\end{equation*}
 where $K$ is a positive constant depending on $p.$
 \end{lemma}
\begin{proof}
%
    Consider the system of equations
\begin{equation*}
  \curl (A\ \curl E) + E = \curl f + g
\end{equation*} 
where $g \in{L^{q'}(\Omega)}$ and
 $f \in{L^{q'}(\Omega)}$ for
 $q' \leq 2.$\\
Let $f_k $ and $g_k, k=1,2,....$ be a sequence of vector fields in $L^2(\Omega)$ such that
\begin{align}
& {\Vert f_k - f\Vert}_{L^{q'}(\Omega)}\rightarrow 0 \hspace{.25cm} as \hspace{.25cm} k\rightarrow \infty,\nonumber \\
& {\Vert g_k - g\Vert}_{L^{q'}(\Omega)}\rightarrow 0 \hspace{.25cm} as \hspace{.25cm} k\rightarrow \infty. \nonumber
\end{align}
Thus, from the $L^2$-theory of Maxwell equation, given $g_k \in L^2(\Omega)$ and $ f_k \in L^2(\Omega)$, there exists unique $E_k \in H_0(\curl;\Omega)$
such that
\begin{equation*}
 \int_{\Omega} A(x)\ \curl E_k(x)\cdot \curl F(x)dx + \int_{\Omega} E_k(x)\cdot F(x)dx = \int_{\Omega} f_k(x)\cdot \curl F(x)dx + \int_{\Omega} g_k(x)\cdot F(x)dx
\end{equation*}
i.e.
\begin{equation}\label{5.18}
 {\mathcal{B}}_{A}(F,E_k) = \int_{\Omega} \curl F(x)\cdot f_k(x)dx + \int_{\Omega} F(x)\cdot g_k(x)dx
\end{equation}
holds for all $F \in H_0(\curl; \Omega) = H_{0}^{1,2}(\curl; \Omega).$
Since $H_{0}^{1,q}(\curl; \Omega) \subset H_{0}^{1,2}(\curl; \Omega) \subset H_{0}^{1,{q'}}(\curl; \Omega)$
it follows from the assumption \eqref{5.16} and from \eqref{5.18} that
\begin{equation}\label{5.19}
 {\Vert E_k\Vert}_{L^{q'}(\Omega)} + {\Vert \curl E_k\Vert}_{L^{q'}(\Omega)} \leq K\{ {\Vert f_k\Vert}_{L^{q'}(\Omega)} + {\Vert g_k\Vert}_{L^{q'}(\Omega)}\}.
\end{equation}
Since, $H_{0}^{1,{q'}}(\curl; \Omega)$ is a reflexive space\footnote{ 
One way to justify this property is as follows.
We define the embedding map $i : H_{0}^{1,{q'}}(\curl; \Omega) \rightarrow L^{q'}(\Omega)\times L^{q'}(\Omega)$ by
 $E\mapsto (E,\curl E)$. Since $i$ is an isometric linear map and $ H_{0}^{1,{q'}}(\curl; \Omega)$ is complete then $i(H_{0}^{1,{q'}}(\curl; \Omega))$ is a closed subspace of $L^{q'}(\Omega)\times L^{q'}(\Omega)$.
As $L^{q'}(\Omega)\times L^{q'}(\Omega)$  is reflexive, see [\cite{Simader}, Lemma 4.3] for instance, and since a closed subspace of a reflexive space is itself a reflexive,
therefore $H_{0}^{1,{q'}}(\curl; \Omega)$ is a reflexive space.} and $\{E_k\}$ is bounded in $H_{0}^{1,{q'}}(\curl; \Omega),$
then $\{E_k\}$ has a weakly convergent subsequence in $H_{0}^{1,{q'}}(\curl; \Omega)$,
i.e, there exists $ E\in H_{0}^{1,{q'}}(\curl; \Omega)$ such that
\begin{equation*}
 E_{k_r}\rightharpoonup E  \hspace{.25cm} in \hspace{.25cm} H_{0}^{1,{q'}}(\curl; \Omega).
\end{equation*}
Therefore, we obtain
\begin{equation}\label{weakly}
 \begin{split}
 {\Vert E\Vert}_{1,{q'}} 
&\leq \liminf_{k_r \rightarrow \infty} {\Vert {E_{k_r}}\Vert}_{1,{q'}} \\
& \leq K \liminf_{k_r \rightarrow \infty} \{ {\Vert f_{k_r}\Vert}_{L^{q'}(\Omega)} +{\Vert g_{k_r}\Vert}_{L^{q'}(\Omega)}\} \\
& \leq K \{ {\Vert f\Vert}_{L^{q'}(\Omega)} +{\Vert g\Vert}_{L^{q'}(\Omega)}\}. 
 \end{split}
\end{equation}
Hence $E$ is in $H_{0}^{1,{q'}}(\curl; \Omega)$, solves equation \eqref{5.17} and from \eqref{weakly}, 
it satisfies the estimate
\begin{equation}\label{last_eq}
 {\Vert E\Vert}_{L^{q'}(\Omega)} + {\Vert \curl E\Vert}_{L^{q'}(\Omega)} \leq K \{ {\Vert f\Vert}_{L^{q'}(\Omega)} +{\Vert g\Vert}_{L^{q'}(\Omega)}\}.
\end{equation}
In addition, suppose that $E$ is a function in $ H_{0}^{1,{q'}}(\curl; \Omega)$ and solves equation \eqref{5.17} with $f=0 $ and $g=0$.
Then from \eqref{last_eq} we see that $E=0.$ This shows that the solution is unique. 
\end{proof} 
Now, we state the main result of this section.
\begin{theorem}\label{5.6}
 Let $\Omega$ be a bounded $C^1$ domain in ${\mathbb R}^3$. Consider the system of differential equations
\begin{equation}\label{5.21}
 \curl (A\ \curl E) + E = \curl f + g,
\end{equation}
where $A = A(x)$ is a Hermitian matrix, with measurable entries, and satisfies the uniform ellipticity condition \eqref{unii}. Then for every  
$f \in L^{p}(\Omega)$ and $g \in L^{p}(\Omega)$, there exists $\delta >0$ such that the problem \eqref{5.21} has a unique weak solution in $H_{0}^{1,{p}}(\curl; \Omega)$, 
where $p \in {(\frac{2+\delta}{1+\delta}, 2 ]}$. In addition, the solution satisfies the estimate
\begin{equation}\label{5.22}
  {\Vert E\Vert}_{L^{p}(\Omega)} + {\Vert \curl E\Vert}_{L^{p}(\Omega)} \leq C \{ {\Vert f\Vert}_{L^{p}(\Omega)} +{\Vert g\Vert}_{L^{p}(\Omega)}\},
\end{equation}
where $C$ is a constant depending only on $\Omega, \lambda $ and $p.$
\end{theorem}
\begin{proof}
 Rewrite
\begin{align*}
 \mathcal{B}_A(E,F)
& = \int_{\Omega} (A(x) \curl E(x)\cdot \curl F(x))dx + \int_{\Omega} E(x)\cdot F(x)dx \nonumber \\
& = [\int_{\Omega} \curl E(x)\cdot \curl F(x)dx + \int_{\Omega} E(x)\cdot F(x)dx] + \int_{\Omega} (A-I)(x) \curl E(x)\cdot \curl F(x)dx \nonumber \\
& = \mathcal{B}(E,F) + {\tilde {\mathcal{B}}}_A (E,F) \nonumber
\end{align*}
where, ${\tilde {\mathcal{B}}}_A (E,F) = \int_{\Omega} (A-I)(x) \curl E(x)\cdot \curl F(x)dx.$
Hence
\begin{equation}
 \sup_{{\Vert E\Vert}_{1,p'} = 1}{\vert \mathcal{B}_A \vert} \geq \sup_{{\Vert E\Vert}_{1,p'} = 1}{\vert \mathcal{B} \vert} - \sup_{{\Vert E\Vert}_{1,p'} = 1}{\vert {\tilde {\mathcal{B}}}_A \vert},
\end{equation}
where $2 \leq p' <\infty,$ $ F$ is in $H_{0}^{1,{p}}(\curl; \Omega)$ and $E$ varies over $H_{0}^{1,p'}(\curl; \Omega).$ Then from
Lemma \ref{5.5}, applied for $p'$, we have
\begin{equation*}
 \sup_{{\Vert E\Vert}_{1,p'} = 1} \vert \mathcal{B}\vert \geq {\frac{1}{K_{p'}}} {\Vert F\Vert}_{1,p}.
\end{equation*}
Also, we have
\begin{equation*}
\sup_{{\Vert E\Vert}_{1,p'} = 1} \vert {\tilde {\mathcal{B}}}_A \vert \leq (1-\lambda) {\Vert F\Vert}_{1,p}
\end{equation*}
where the constant $K_p$ is defined as $K_p = {\Vert T_I \Vert}_p $ and the constant $\lambda$ is defined in \eqref{uni}. Therefore
\begin{equation}
 \inf_{{\Vert F\Vert}_{1,p} = 1} \sup_{{\Vert E\Vert}_{1,p'} = 1} \vert \mathcal{B}(E,F)\vert \geq (\frac{1}{K_{p'}} - 1 +\lambda).
\end{equation}
From the Riesz Convexity Theorem, $T_I$ is a bounded linear operator from $L^p(\Omega) \times L^p(\Omega) \rightarrow L^p(\Omega) \times L^p(\Omega)$ for every $p$ in the range $1 \leq p < \infty$
and $\log {\Vert T_I\Vert}_p$ is a convex function of $\frac{1}{p}$, see \cite{Bergh}.
In particular, the function $p \longmapsto {\Vert T_I\Vert}_p$ is continuous for $p \in [1, \infty)$.
Now, define $F : [1, \infty) \rightarrow \mathbb{R}$ by $F(p) := \frac{1}{K_p} -1 +\lambda$.
As  $F(2)=\lambda > 0$ and $F(p)$ is continuous, then there exists $\delta > 0$ such that $F(p) > 0$ for all $p \in (2 - \delta, 2 + \delta).$
Thus, in the interval $2 \leq p' < 2 + \delta$ we have 
\begin{equation}
 \inf_{{\Vert F\Vert}_{1,p} =1} \sup_{{\Vert E\Vert}_{1,p'}=1} {\vert {\mathcal{B}}_{A}(E,F)\vert}\geq \frac{1}{{\tilde K}_{p'}} >0
\end{equation}
where ${\tilde K}_{p'} = \frac{K_{p'}}{1- K_{p'}(1 - \lambda)}$. The rest is a consequence of Lemma \ref{5.3} (taking $q' = p$).   
\end{proof}
\begin{remark}
To deal with the case $p>2$, we need the corresponding result to Lemma \ref{5.3} for $q<2$. So far, the technique in \cite{Me} does not
go for the Maxwell system as naturally as for the case $q>2$ described in the proof of Lemma \ref{5.3}. For $p>2$, the approach by Gr{\"o}ger, see \cite{Groe}, might be the correct one. 
The details will be given in a forthcoming work. 
\end{remark}
\begin{section}{\textbf{Appendix}}
Here, we recall some important properties from \cite{Mitrea} concerning the vector layer potentials and some Sobolev spaces useful for 
the study of the problems related to the Maxwell system.\\
\textbf{The Surface Divergence }\\
    Let $\Omega\subset\mathbb{R}^3$ be a bounded Lipschitz domain and $A\in{L^p_{\tan}}(\partial\Omega),$
 we define $\Div\;A$ as a linear functional
\begin{equation}\label{6.3}
 \langle\Div A,f\rangle:= -\int_{\partial\Omega}A(x)\cdot\nabla_{\tan}f(x)ds(x)
\end{equation}
where $f$ is any Lipschitz continuous function on ${\partial\Omega}$ and the space ${L^p_{\tan}}(\partial\Omega)$ is defined as
\begin{equation*}
 {L^p_{\tan}}(\partial\Omega) := \{A : {\partial\Omega}\rightarrow{\mathbb{C}^3}; A\in L^p(\partial\Omega), (\nu, A) = 0 \hspace{.15cm} a.e. \hspace{.15cm} \mbox{on} \hspace{.15cm}{\partial\Omega}\},
\end{equation*}
where $1<p<\infty$ and the usual tangential gradient $\nabla_{\tan} := -\nu{\times(\nu\times{\nabla})}.$\\ 
In the following theorem we gather some of the key properties which are useful in our analysis in the previous sections.
\begin{theorem}\label{from Mitrea}  
Let $D $ be a bounded Lipschitz domain and $\Omega$ be a bounded $C^1$ domain in $\mathbb{R}^3$ and let $p$ be real number such that  $1<p<\infty$. 
Then we have the following properties:
\begin{property}\label{pr5}
The operator $S_k$ is a compact operator on $L^p(\partial D)$. Also, the operator $\mathcal{S}_k$ maps $W^{-\frac{1}{p},p}(\partial D)$ 
boundedly into $W^{1,p}(D)$ (and also into $W^{1,p}(\Omega\setminus\overline{D})$ if $\overline{D}\subset\Omega$). In particular, $\nu\cdot\nabla \mathcal{S}_k = -\frac{1}{2}I + K_{k}^{*}$ on $W^{-\frac{1}{p},p}(\partial D).$ 
\end{property}
\begin{property}\label{pr4}
Let $u \in L^{p}(D)$ such that $\curl u \in L^{p}(D)$. If ${\nu\wedge u} \in L^{p}(\partial D), $ then in
fact $\nu\wedge u \in {L^{p}_{tan}}(\partial D)$ and $\Div (\nu\wedge u) = -\nu\cdot{\curl u}.$ In particular,
$\Div (\nu\wedge u) \in W^{-{\frac{1}{p}},p}(\partial D).$ 
\end{property}
\begin{property}\label{pr6}
For $k\in \mathbb{C},$ and $A\in L_{tan}^{p}(\partial D)$, we have $\div \mathcal{S}_kA = \mathcal{S}_k(\Div A)$ in $\mathbb{R}^3\setminus \partial D.$ 
\end{property}
\begin{property}\label{pr3}
 If the vector field $u \in{L^p(D)}$ is such 
that $\curl u \in{L^p(D)},$ then $\nu\wedge u \in{W^{-{\frac{1}{p}},p}(\partial D)}$ and there exists a positive constant C depending only
on the Lipschitz character of $\partial D$ so that
\begin{equation*}
 {\|{\nu\wedge u}\|}_{W^{-{\frac{1}{p}},p}(\partial D)} \leq C \{{\| {u}\|}_{L^{p}(D)}+{\|\curl u\|}_{L^{p}(D)}\}.
\end{equation*}
Also, if $u\in{L^p(D)}$ is such that $\div u \in{L^p(D)},$ then $\nu\cdot u \in{W^{-{\frac{1}{p}},p}(\partial D)}$ and
\begin{equation*}
 {\|{\nu\cdot u}\|}_{W^{-{\frac{1}{p}},p}(\partial D)} \leq C \{{\|{u}\|}_{L^{p}(D)}+{\|\div u\|}_{L^{p}(D)}\}
\end{equation*}
for some positive C depending only on the Lipschitz character of $\partial D.$
\end{property}
\begin{property}\label{pr7}
Consider a vector field $u\in L^p(D)$ such that $\div u \in L^p(D)$ and $\curl u \in L^p(D).$ 
If  $\nu\wedge u \in L^p(\partial D),$ then also $\nu\cdot u\in L^p(\partial D)$, for $p \in (1, \infty)$. If in addition $1<p\leq 2$, 
then $u \in B_{\frac{1}{p}}^{p,2}(D)$ and we have the estimate
\[
 \| u\|_{B_{\frac{1}{p}}^{p,2}(D)} \leq C \{ \|u\|_{L^p(D)} + \|\curl u\|_{L^p(D)} + \|\nabla\cdot u\|_{L^p(D)} +\|\nu\wedge u\|_{L^p(\partial D)}\}
\]
where the Sobolev-Besov space $B_{\alpha}^{p,q}(D) := [L^p(D),W^{1,p}(D)]_{\alpha,q}$ is obtained
by real interpolation for $1 < p,q < \infty$ and $0 < \alpha < 1.$ 
\end{property}
\begin{property}\label{pr2}
 Assume that $k\in \mathbb{C}\setminus \{0\}$ with $Im\ k \geq 0$ is not a Maxwell eigenvalue for $\Omega.$ 
Then for given $K\in L^p(\Omega)$ and $J\in L^p(\Omega)$, the following problem
 \begin{equation}
\begin{cases} 
&\operatorname{\curl}\ E- ik H= K \ \ \text{in}\ \Omega,\\
&\curl H + ik E = J \ \ \text{in}\ \Omega,\\
& \nu\wedge E=0 \ \ \text{on}\ \partial\Omega,
\end{cases}
\end{equation}
has a unique solution and there exists $C= C(p,k,\Omega)>0$ such that
\[
 \|E\|_{L^p(\Omega)} + \|H\|_{L^p(\Omega)} \leq C\{ \|K\|_{L^p(\Omega)} + \|J\|_{L^p(\Omega)}\} 
\]
for $1 < p < \infty.$ 
\end{property}
\end{theorem}
 The properties \ref{pr5} (except the boundedness of $\mathcal{S}_k$ from $W^{-1/p,p}(\partial D)$ into $W^{1,p}(\Omega\setminus\overline{D})$),
 \ref{pr4}, \ref{pr6}, \ref{pr3}, \ref{pr7} and \ref{pr2} are proved as Lemma 3.1, Lemma 4.1, Lemma 4.2, Lemma 2.3, Corollary 10.3 and Theorem 11.6 in \cite{Mitrea} respectively. 
For the shake of completeness, we show the boundedness of $\mathcal{S}_k : W^{-1/p,p}(\partial D) \rightarrow W^{1,p}(\Omega\setminus\overline{D}),$ for $1<p<\infty$.\\
Indeed, recall that
 \[
  \mathcal{S}_kf(x) = \int_{\partial D}\Phi_k(x,y)f(y)ds(y), \ \ x\in\mathbb{R}^3\setminus\partial D.
 \]
As the trace map $\gamma : W^{1,p}(\Omega\setminus\overline{D}) \rightarrow W^{1-1/p,p}(\partial\Omega\cup\partial D)$
is bounded and has a bounded right inverse then
\begin{equation}\label{biriyani}
 \|\mathcal{S}_kf\|_{W^{1,p}(\Omega\setminus\overline{D})}
 \leq C [\|\mathcal{S}_kf\|_{W^{1-1/p,p}(\partial D)} + \|\mathcal{S}_kf\|_{W^{1-1/p,p}(\partial\Omega)}].
\end{equation}
Note that, $\mathcal{S}_k : W^{-1/p,p}(\partial D) \rightarrow W^{1,p}(D)$ is bounded, see Property \ref{pr5} of Theorem \ref{from Mitrea} or [\cite{Mitrea}, Lemma 3.1], and also
the trace map $\gamma : W^{1,p}(D) \rightarrow W^{1-1/p,p}(\partial D)$ is bounded. Therefore
\begin{align}
 \|\mathcal{S}_kf\|_{W^{1-1/p,p}(\partial D)}
 & \leq C \|\mathcal{S}_kf\|_{W^{1,p}(D)} \nonumber \\
 & \leq C \|f\|_{W^{-1/p,p}(\partial D)} \label{mutton}
\end{align}
On the other hand
\begin{align*}
 \|\mathcal{S}_kf\|_{W^{1-1/p,p}(\partial\Omega)}^{p}
 & = \int_{\partial\Omega}|\mathcal{S}_kf(x)|^pds(x) + \int_{\partial\Omega}\int_{\partial\Omega}\frac{|\mathcal{S}_kf(x)-\mathcal{S}_kf(y)|^p}{|x-y|^{2+p(1-1/p)}}ds(x)ds(y) \\
 & =: I + II.
\end{align*}
 Since, $c_1<|x-y|<c_2$ for $x\in\partial\Omega$ and $y\in\partial D$, therefore
\begin{align*}
 |I|
  \leq \int_{\partial\Omega}\left(\int_{\partial D}\frac{1}{|x-y|}|f(y)|ds(y)\right)^pds(x) 
 & \leq \frac{1}{c_1}\int_{\partial\Omega}\left(\int_{\partial D}|f(y)|ds(y)\right)^pds(x) \\
 & \leq C \|f\|_{W^{-1/p,p}(\partial D)}^{p}.
\end{align*}
Using the Taylor expansion of $e^{ikt}, t\in\mathbb{R},$ we can show that
for $x,y\in\partial\Omega$ and $z\in\partial D$ where $D\subset\Omega$, we have
\begin{equation}\label{fUUUnda}
 |\frac{e^{ik|x-z|}}{|x-z|}-\frac{e^{ik|y-z|}}{|y-z|}|\leq C|x-y|.
\end{equation}
Note that
\begin{align*}
 |\mathcal{S}_kf(x)-\mathcal{S}_kf(y)|
 & = |\int_{\partial D}[\Phi_k(x,z)-\Phi_k(y,z)]f(z)ds(z)|\\
& = |\int_{\partial D}[\frac{1}{4\pi}\left(\frac{e^{ik|x-z|}}{|x-z|}-\frac{e^{ik|y-z|}}{|y-z|}\right)]f(z)ds(z)|\\
& (\text{using}\ \eqref{fUUUnda})  \\
& \leq C |x-y|\int_{\partial D}|f(z)|ds(z) \\
& \leq C |x-y|\|f\|_{W^{-1/p,p}(\partial D)}. 
\end{align*}
Therefore
\begin{align*}
 |II|
 & \leq C \|f\|_{W^{-1/p,p}(\partial D)}^{p}\int_{\partial\Omega}\int_{\partial\Omega}\frac{|x-y|^p}{|x-y|^{2+p(1-1/p)}}ds(x)ds(y) \\
 & \leq C  \|f\|_{W^{-1/p,p}(\partial D)}^{p}\int_{\partial\Omega}\int_{\partial\Omega}\frac{1}{|x-y|}ds(x)ds(y) \\
 & \leq C  \|f\|_{W^{-1/p,p}(\partial D)}^{p}.
\end{align*}
Hence 
\begin{equation}\label{german}
 \|\mathcal{S}_kf\|_{W^{1-1/p}(\partial\Omega)} \leq C \|f\|_{W^{-1/p,p}(\partial D)}.
\end{equation}
Combining \eqref{biriyani}, \eqref{mutton} and \eqref{german}, we obtain
\[
 \|\mathcal{S}_kf\|_{W^{1,p}(\Omega\setminus\overline{D})} \leq C \|f\|_{W^{-1/p,p}(\partial D)}.
\]
 \end{section}

\end{section}

\end{document}